\documentclass[11pt,a4paper,reqno]{amsart}
\usepackage[T1]{fontenc}
\usepackage{amssymb}
\usepackage[margin=1in]{geometry}
\usepackage{enumitem}
\usepackage{xcolor}
\usepackage[hidelinks]{hyperref}
\numberwithin{equation}{section}

\theoremstyle{plain}
\newtheorem{theorem}{Theorem}[section]
\newtheorem{lemma}[theorem]{Lemma}
\newtheorem{proposition}[theorem]{Proposition}
\newtheorem{corollary}[theorem]{Corollary}
\newtheorem{maintheorem}{Theorem}

\theoremstyle{definition}
\newtheorem{definition}[theorem]{Definition}
\newtheorem{example}[theorem]{Example}
\newtheorem{question}[theorem]{Question}
\theoremstyle{remark}
\newtheorem{remark}[theorem]{Remark}

\newcommand{\N}{\mathbb N}
\newcommand{\R}{\mathbb R}
\newcommand{\K}{\mathcal K}
\newcommand{\dimH}{\dim_{\mathrm H}}
\newcommand{\dimP}{\dim_{\mathrm P}}
\newcommand{\udimB}{\overline{\dim}_{\mathrm B}}
\newcommand{\ldimB}{\underline{\dim}_{\mathrm B}}
\DeclareMathOperator{\diam}{diam}
\DeclareMathOperator{\dist}{dist}
\DeclareMathOperator{\supp}{supp}

\title[The packing spectrum: interpolation between Hausdorff and packing dimensions]{The packing spectrum: interpolation between Hausdorff and packing dimensions}

\author[Peizhi Liu]{Peizhi Liu}
\address{School of Mathematics and Statistics,  Nanjing University of Science and Technology, Nanjing 210094, China}
\email{liupeizhi@njust.edu.cn}

\date{}

\begin{document}
\begin{abstract}

We introduce a family of metric outer measures whose critical exponents
form a packing spectrum interpolating between Hausdorff and packing
dimensions. For analytic sets, we characterize each spectral value
as the supremum of the Hausdorff dimension of a product with a
compact testing set minus the testing set's upper intermediate
dimension at the complementary parameter. The associated product
bounds refine the classical Hausdorff--packing inequalities and can
give exact Hausdorff additivity when both factors have distinct
endpoint dimensions. Below the packing endpoint, we prove a
variational principle in terms of windowed local mass exponents
and sharp parameter estimates implying local Lipschitz continuity;
a jump at the endpoint is possible. We characterize all attainable
profiles by monotonicity, continuity below the packing endpoint,
and a sharp upper right Dini-derivative bound. Every admissible
profile is realized by a compact binary digit set whose uniform
measure attains the finite-window and classical packing variational
suprema. Complementary digit positions yield reflected upper
intermediate-dimension profiles and a single factor attaining
the Hausdorff product bounds at every parameter, with an Ahlfors
regular product measure.
\end{abstract}

\maketitle

\section{Introduction}
\label{sec:introduction}

Hausdorff dimension need not be additive under Cartesian products.
For nonempty compact sets $E\subset\R^m$ and $F\subset\R^n$, write
\[
 \Delta(E,F)=\dimH(E\times F)-\dimH E-\dimH F.
\]
The classical Hausdorff--packing inequalities give
\[
 0\le\Delta(E,F)
 \le\min\{\dimP E-\dimH E,\,\dimP F-\dimH F\};
\]
see~\cite{Tricot1982,Howroyd1996}.
These bounds depend only on the endpoint dimensions of the factors.
We seek finer bounds that record their behavior over intermediate
ranges of scales, and a dimension invariant characterized by the
resulting product estimates.

We introduce a packing spectrum $\dim_\theta^{\mathrm P}$,
$0\le\theta\le1$, whose endpoints are Hausdorff and packing
dimensions. Its definition uses metric outer measures obtained by
covering a set with carriers and measuring packings within each
carrier over a prescribed scale window. Matching this window with
the covering window for upper intermediate dimension yields
Hausdorff product bounds and a converse characterization by compact
testing sets.

Scale restrictions already underlie several dimension spectra.
The intermediate dimensions of Falconer, Fraser and
Kempton~\cite{FFK2020} restrict the diameters in Hausdorff covers,
whereas the Assouad spectrum restricts the relation between the
scales in local covering estimates~\cite{FraserYu2018,FraserBook2020}.
For surveys, see~\cite{FraserInterpolating2021,FalconerSurvey2021}.
Our construction interpolates between Hausdorff and packing
dimensions at the level of metric outer measures and retains
countable stability at every parameter.

\subsection*{The packing spectrum}
\begin{definition}
\label{def:packing-spectrum}
A \emph{labelled carrier} is a pair $(E,R)$ with
$E\subset\mathbb R^d$, $0<R<1$, and $\diam E\le R$.
For $s\ge0$, $\theta\in[0,1]$, and a nonempty labelled carrier $(E,R)$,
define
\begin{equation}\label{eq:local-labelled}
\Phi^{s}_{\theta}(E,R)
=
\sup\left\{
\sum_{i=1}^{N}(2r_i)^s:
\begin{array}{l}
x_i\in\overline E,\quad R^{1/(1-\theta)}\le r_i\le R,\\
|x_i-x_j|\ge r_i+r_j\quad(i\ne j)
\end{array}
\right\}.
\end{equation}
At $\theta=1$, replace the radius condition by $0<r_i\le R$.
Set $\Phi^s_\theta(\varnothing,R)=0$.
For $F\subset\mathbb R^d$ and $0<\delta<1$, define
\begin{equation}\label{eq:outer-labelled}
\mathcal P^{s}_{\theta,\delta}(F)
=\inf\left\{
\sum_j \Phi^s_\theta(E_j,R_j):
F\subset\bigcup_j E_j,\quad
\diam E_j\le R_j\le\delta,\quad R_j>0
\right\}.
\end{equation}
Set
\[
\mathcal P^{s}_{\theta}(F)
=\lim_{\delta\downarrow0}\mathcal P^{s}_{\theta,\delta}(F),
\qquad
\dim_\theta^{\mathrm P}F
=\inf\{s>0:\mathcal P^{s}_{\theta}(F)=0\}.
\]
The function $\theta\mapsto\dim_\theta^{\mathrm P}F$ is the
\emph{packing spectrum} of $F$.
\end{definition}

At $\theta=0$, the window consists of the single radius $R$.
For $0<\theta<1$, it is $[R^{1/(1-\theta)},R]$; at $\theta=1$,
all radii up to the label are allowed. The last case is the
labelled packing construction of~\cite{Liu2026Labelled}.
That work supplies the unrestricted-radius foundation and packing
endpoint; the present paper develops the finite-window theory.

\begin{maintheorem}
\label{thm:intro-structure}
For every $s\ge0$ and $\theta\in[0,1]$, $\mathcal P^s_\theta$ is a
metric outer measure. For $s>0$,
\[
 \mathcal P^s_0=2^s\mathcal H^s,
 \qquad \mathcal P^s\le\mathcal P^s_1\le M_d\mathcal P^s,
\]
where $\mathcal H^s$ is Hausdorff measure with the
diameter-cost normalization, $\mathcal P^s$ is classical packing measure with cost $(2r)^s$
and $M_d$ depends only on $d$. The packing spectrum is countably stable, bi-Lipschitz
invariant, and satisfies the usual H\"older image bound. The
spectrum is locally Lipschitz on $[0,1)$, although a jump at $1$ is
possible. For analytic sets, the spectrum is determined by compact
subsets.
For each $0\le\theta<1$, the map
$K\mapsto\dim_\theta^{\mathrm P}K$ on the hyperspace of nonempty
compact sets, equipped with the Hausdorff metric, is of Baire class
at most $2$.
\end{maintheorem}

In particular,
\[
 \dim_0^{\mathrm P}E=\dimH E,\qquad
 \dim_1^{\mathrm P}E=\dimP E.
\]
The spectrum is nondecreasing in $\theta$.
More precisely, if $f(\theta)=\dim_\theta^{\mathrm P}E$ and
$0\le\theta_1<\theta_2<1$, then
\[
 f(\theta_1)\le f(\theta_2)
 \le \frac{d f(\theta_1)}
 {c d+(1-c)f(\theta_1)},\qquad
 c=\frac{1-\theta_2}{1-\theta_1}.
\]
This is Theorem~\ref{thm:quantitative-continuity}; it gives both
the stated regularity and the necessary profile constraint below.
The remaining properties are established in
Section~\ref{sec:basic-properties}.

\subsection*{A measure-theoretic characterization}
The packing spectrum admits a local mass description, which will
also be used in the compact testing construction below. For a compactly supported Borel
probability measure $\mu$, let $B(x,r)$ be the open ball and put
\[
 d_\theta^\mu(x)
 =\liminf_{R\downarrow0}
   \sup_{R^{1/(1-\theta)}\le r\le R}
   \frac{\log\mu(B(x,r))}{\log r},
 \qquad x\in\supp\mu,\quad 0\le\theta<1.
\]
At $\theta=0$ this is the lower local dimension.
The classical packing endpoint is treated separately.

\begin{maintheorem}
\label{thm:intro-variational}
For every nonempty analytic set $E\subset\R^d$ and every
$0\le\theta<1$,
\[
 \dim_\theta^{\mathrm P}E
 =\sup_{\substack{\mu\text{ probability}\\\supp\mu\Subset E}}
   \operatorname*{ess\,inf}_{x\sim\mu}d_\theta^\mu(x),
\]
where $\supp\mu\Subset E$ means that the support is a compact subset
of $E$. For $0<s<\dim_\theta^{\mathrm P}E$, one such
measure satisfies, for almost every $x$ and every $0<v<s$,
\[
 \forall\,0<R<R_0(x,v)\quad
 \exists\,r\in[R^{1/(1-\theta)},R]:
 \qquad \mu(B(x,r))\le r^v.
\]
\end{maintheorem}

Theorem~\ref{thm:main} and \eqref{eq:suitable-radius} give these
assertions. For each fixed parameter and subcritical exponent,
one measure works at almost every point and in every sufficiently
small window. The suitable radius can depend on the point and
the window.

This order of quantifiers distinguishes the construction from
related intermediate dimensions. The Frostman lemma for intermediate
dimensions~\cite[Proposition~2.3]{FFK2020} uses a scale-dependent
measure controlling all radii in a window.
Angelini's packing intermediate dimensions, defined for measures
and then for analytic sets, instead require subcritical mass decay
at every radius in arbitrarily small windows, almost everywhere
\cite[Theorem~4.3 and Definition~8.2]{Angelini2026c}.
The resulting set invariant is different from the packing spectrum:
Example~\ref{ex:comparison-constructions} gives explicit separations
in both directions. Constructions restricting covering diameters
or packing radii to a fixed global set of scales
\cite{BalkaKeleti2025,Zindulka2013} provide another comparison;
here the window varies with each carrier label.

\subsection*{Hausdorff products and compact testing sets}
For a bounded set $F$ and $0<\eta\le1$, the upper intermediate
dimension $\overline{\dim}_\eta F$ is defined using covers with
diameters in $[R^{1/\eta},R]$, at every sufficiently small outer
scale $R$~\cite{FFK2020}. We use the endpoint conventions
$\overline{\dim}_0F=\dimH F$ and
$\overline{\dim}_1F=\udimB F$.
The countably stable modification
$\overline{\dim}^{\mathrm{mod}}_\eta$ is obtained by minimizing
the supremum of $\overline{\dim}_\eta$ over countable bounded
decompositions; its endpoints are $\dimH$ and $\dimP$.
The definitions are given in Section~\ref{sec:geometric-properties}.

\begin{maintheorem}
\label{thm:intro-products}
For $0\le\theta\le1$ and bounded sets $E\subset\R^m$ and $F\subset\R^n$,
\begin{equation}\label{eq:intro-product-bound}
 \dimH(E\times F)
 \le\min\{\dim_\theta^{\mathrm P}E+\overline{\dim}_{1-\theta}F,
             \dim_\theta^{\mathrm P}F+\overline{\dim}_{1-\theta}E\}.
\end{equation}
The same bound holds for arbitrary sets with
$\overline{\dim}^{\mathrm{mod}}_{1-\theta}$, the countably stable
modification defined in Section~\ref{sec:geometric-properties}.
Conversely, for every nonempty analytic $E\subset\R^m$,
\begin{equation}\label{eq:intro-product-characterization}
 \dim_\theta^{\mathrm P}E
 =\sup_{\substack{\varnothing\ne G\subset[0,1]^m\\G\text{ compact}}}
   \{\dimH(E\times G)-\overline{\dim}_{1-\theta}G\}.
\end{equation}
The same characterization holds with
$\overline{\dim}^{\mathrm{mod}}_{1-\theta}G$ in place of
$\overline{\dim}_{1-\theta}G$.
\end{maintheorem}

At $\theta=0$, the first term in
\eqref{eq:intro-product-bound} is $\dimH E+\udimB F$; at
$\theta=1$ it is $\dimP E+\dimH F$.
With the modified intermediate dimension, the two endpoints
recover the classical Hausdorff--packing bounds in opposite
orders. In terms of the excess introduced above,
\[
 \Delta(E,F)\le\inf_{0\le\theta\le1}
 \{\dim_\theta^{\mathrm P}E-\dimH E
   +\overline{\dim}^{\mathrm{mod}}_{1-\theta}F-\dimH F\},
\]
and the same estimate holds with the factors exchanged.
This improvement can be strict. Example~\ref{ex:strict-product-improvement}
gives compact sets with
\[
 \dimH E=\dimH F=\tfrac14,\qquad
 \dimP E=1,\qquad \dimP F=\tfrac12,
 \qquad \dimH(E\times F)=\tfrac12.
\]
The classical upper bound is $3/4$, whereas the window bounds
converge to $1/2$. Thus exact additivity can be detected even
when both factors have distinct Hausdorff and packing dimensions.

The converse \eqref{eq:intro-product-characterization} identifies
the spectrum as the optimal value of these product tests.
At $\theta=1$, it recovers the packing-dimension characterization
of Bishop and Peres~\cite{BishopPeres1996}; see also
Xiao~\cite[Theorem~3.1]{Xiao1996} for Borel subsets of $\R$.
Our compact endpoint construction uses the additive-complement
theorem of Keleti and M\'ath\'e~\cite[Theorem~1.9]{KeletiMathe2023}.

The spectrum of a product itself also satisfies
Theorem~\ref{thm:spectrum-product-lower}: for nonempty analytic
$E\subset\R^m$, $F\subset\R^n$, and $0\le\theta\le1$,
\[
 \begin{split}
 \max\{\dim_\theta^{\mathrm P}E+\dimH F,\,
          \dimH E+\dim_\theta^{\mathrm P}F\}
 &\le\dim_\theta^{\mathrm P}(E\times F)\\
 &\le\min\{\dim_\theta^{\mathrm P}E+\dimP F,\,
              \dim_\theta^{\mathrm P}F+\dimP E\}.
 \end{split}
\]

\subsection*{Attainable profiles and complementary digit sets}
The finite packing window imposes a sharp restriction on the
variation of the spectrum. This restriction also characterizes
all possible profiles.

\begin{maintheorem}
\label{thm:intro-realization}
Let $d\ge1$ be an integer. A function $f:[0,1]\to[0,d]$ is the
packing spectrum of a nonempty bounded subset of $\R^d$ if and
only if it is nondecreasing, continuous on $[0,1)$, and satisfies
\[
 0\le D^+f(\theta)
 \le\frac{f(\theta)(d-f(\theta))}{d(1-\theta)}
 \qquad(0\le\theta<1),
\]
where $D^+$ is the upper right Dini derivative.
\end{maintheorem}

Here
$D^+f(\theta)=\limsup_{h\downarrow0}
(f(\theta+h)-f(\theta))/h$.
The necessity follows from the quantitative parameter estimate.
For sufficiency, we use the homogeneous Moran realization of
Banaji and Rutar~\cite[Theorem~3.9]{BanajiRutar2022}, transfer its
upper intermediate profile and Hausdorff endpoint to binary digit
sets, and pass to complementary digit positions
(Lemma~\ref{lem:moran-digit-transfer}).
Equivalently, $h(\eta)=d-f(1-\eta)$ is an attainable upper
intermediate-dimension profile, with its Hausdorff value at
$\eta=0$ (Corollary~\ref{cor:reflected-profile-classification}).

More precisely, every admissible $f$ is realized by a compact set
$E_S\subset[0,1]^d$ in which binary digits are free at the positions
in $S\subset\N$ and zero elsewhere.
Its complementary set $E_{S^c}$ satisfies
\[
 \overline{\dim}_{1-\theta}E_{S^c}
 =\overline{\dim}^{\mathrm{mod}}_{1-\theta}E_{S^c}
 =d-f(\theta),\qquad 0\le\theta\le1.
\]
The uniform digit measure $\mu_S$ satisfies
$d_\theta^{\mu_S}(x)=f(\theta)$ at every $x\in E_S$ for
$0\le\theta<1$; its classical upper local dimension equals
$f(1)$ at every point (Remark~\ref{rem:digit-profile-realization}).
Thus this one measure attains the finite-window variational
suprema and the classical packing supremum.
The product measure $\mu_S\times\mu_{S^c}$ is $d$-Ahlfors regular
on $E_S\times E_{S^c}$ (Lemma~\ref{lem:complementary-digits}).
The associated sumset and Lebesgue convolution identities are
given in \eqref{eq:complementary-digit-convolution}.

Consequently, every admissible packing profile is the spectrum of a
coordinate projection of a $d$-Ahlfors regular compact set in
$\R^{2d}$. One complementary factor attains the Hausdorff
product bound for every parameter, and
\[
 \Delta(E_S,E_{S^c})=f(1)-f(0).
\]
The two contributions to this excess are
$f(\theta)-f(0)$ and $f(1)-f(\theta)$.
Example~\ref{ex:extreme-profiles} exhibits both a jump at the
packing endpoint and a profile attaining the upper bound in the two-parameter
estimate.

Section~\ref{sec:basic-properties} establishes the outer-measure,
endpoint, and parameter properties.
Section~\ref{sec:single-measure-variational-principle} proves the
variational principle, and Section~\ref{sec:geometric-properties}
develops the product bounds and compact tests.
Section~\ref{sec:Complementarity} proves the window formulas and
profile classification. Section~\ref{sec:further-questions}
gives examples, comparisons, and further questions.

\section{Basic properties}
\label{sec:basic-properties}

We write $\ldimB$ and $\udimB$ for lower and upper box dimension,
respectively; see \cite{Falconer2014,Mattila1995}.
Let $\mathcal K(\mathbb R^d)$ be the hyperspace of nonempty compact sets,
equipped with the Hausdorff metric $d_H$.

\begin{proposition}
\label{prop:basic-properties}
Let $d\ge1$ be an integer. The packing spectrum has the following
properties.
\begin{enumerate}[label=\textup{(\roman*)},ref=\textup{(\roman*)}]
\item \label{item:basic-monotonicity}\emph{Monotonicity.}
For $0\le\theta\le1$ and $E\subset F\subset\R^d$,
$\dim_\theta^{\mathrm P}E\le\dim_\theta^{\mathrm P}F$.

\item \label{item:basic-bounds}\emph{Dimension bounds.}
For every $F\subset\R^d$ and $0\le\theta\le1$,
\[
 0\le\dimH F\le\dim_\theta^{\mathrm P}F\le\dimP F\le d.
\]

\item \label{item:basic-mappings}\emph{H\"older and Lipschitz maps.}
Let $F\subset\R^d$, and let $f:F\to\R^m$ satisfy
\[
 |f(x)-f(y)|\le L|x-y|^\alpha\qquad(x,y\in F),
 \qquad 0<\alpha\le1,\quad L\ge1.
\]
Then, for every $0\le\theta\le1$,
\[
 \dim_\theta^{\mathrm P}f(F)\le\alpha^{-1}\dim_\theta^{\mathrm P}F.
\]
In particular, Lipschitz maps do not increase the packing spectrum,
and bi-Lipschitz maps preserve it.

\item \label{item:basic-countable}\emph{Countable stability.}
For every sequence $(F_j)_{j\ge1}$ of subsets of $\R^d$ and every
$0\le\theta\le1$,
\begin{equation}\label{eq:countable-stability}
 \dim_\theta^{\mathrm P}\Bigl(\bigcup_{j=1}^{\infty}F_j\Bigr)
 =\sup_j\dim_\theta^{\mathrm P}F_j.
\end{equation}

\item \label{item:basic-baire}\emph{Baire class.}
For $0\le\theta<1$, the map
$K\mapsto\dim_\theta^{\mathrm P}K$ from $\mathcal K(\R^d)$ to
$[0,d]$ is of Baire class at most $2$.

\item \label{item:basic-analytic}\emph{Compact approximation.}
For every nonempty analytic set $A\subset\R^d$ and every
$0\le\theta\le1$,
\begin{equation}\label{eq:add-analytic-compact-determination}
 \dim_\theta^{\mathrm P}A
 =\sup\{\dim_\theta^{\mathrm P}K:
          \varnothing\ne K\subset A,\ K\text{ compact}\}.
\end{equation}
\end{enumerate}
\end{proposition}

\subsection{Endpoint identities}

\begin{lemma}
\label{lem:elementary-local-estimates}
Let $0<R<1$, $\theta\in[0,1]$, and $0\le s\le t$.  If $E\ne\varnothing$
and $\diam E\le R$, then
\[
 (2R)^s\le \Phi^s_\theta(E,R),\qquad
 \Phi^t_\theta(E,R)\le(2R)^{t-s}\Phi^s_\theta(E,R).
\]
For fixed $R$, $\Phi^s_\theta(E,R)$ is invariant under closure and
translation of $E$, monotone under inclusion of $E$, and nondecreasing
in $\theta$. For every singleton,
\[
 \Phi^s_\theta(\{x\},R)=(2R)^s.
\]
For $\theta<1$, $\Phi^s_\theta(E,R)<\infty$.
\end{lemma}

\begin{proof}
Choosing one centre with radius $R$ gives $\Phi^s_\theta(E,R)\ge(2R)^s$,
with equality for singletons. For every admissible configuration,
\[
 \sum_i(2r_i)^t\le(2R)^{t-s}\sum_i(2r_i)^s;
\]
taking suprema proves the exponent comparison.
Closure and translation preserve admissibility; enlarging $E$ or
increasing $\theta$ enlarges the class of admissible configurations.
If
$\theta<1$, admissible centres are separated by at least
$2R^{1/(1-\theta)}>0$ in the compact set $\overline E$, so their number is
uniformly bounded.
\end{proof}

\begin{proposition}
\label{prop:metric-outer-measure}
\label{lem:Ms-theta-basic-open}
For every $s\ge0$ and $\theta\in[0,1]$, the set function
$\mathcal P^{s}_{\theta}$ is a metric outer measure.  In particular, every
Borel set is $\mathcal P^{s}_{\theta}$-measurable.  If $0\le s\le t$ and
$0<\delta<1$, then
\begin{equation}\label{eq:exponent-comparison}
 \mathcal P^{t}_{\theta,\delta}(F)
 \le(2\delta)^{t-s}\mathcal P^{s}_{\theta,\delta}(F).
\end{equation}
Consequently, if $t>s$ and $\mathcal P^{s}_{\theta}(F)<\infty$, then
$\mathcal P^{t}_{\theta}(F)=0$.
\end{proposition}

\begin{proof}
For fixed $\delta$, monotonicity follows because a labelled cover of
$B$ covers every $A\subset B$. Combining covers and taking infima proves
countable subadditivity, and
$\mathcal P^{s}_{\theta,\delta}(\varnothing)=0$.
For every $F$, the function $\delta\mapsto\mathcal P^{s}_{\theta,\delta}(F)$
is nonincreasing. For any sequence $(F_i)$,
\[
 \mathcal P^{s}_{\theta,\delta}\Bigl(\bigcup_iF_i\Bigr)
 \le\sum_i\mathcal P^{s}_{\theta,\delta}(F_i)
 \le\sum_i\mathcal P^{s}_{\theta}(F_i),
\]
so letting $\delta\downarrow0$ proves countable subadditivity.
Thus $\mathcal P^{s}_{\theta}$ is an outer measure.

If $\dist(A,B)>\delta$, no carrier of diameter at most $\delta$ can meet
both $A$ and $B$.  Splitting a labelled cover therefore gives
\[
 \mathcal P^{s}_{\theta,\delta}(A\cup B)
 \ge\mathcal P^{s}_{\theta,\delta}(A)
   +\mathcal P^{s}_{\theta,\delta}(B).
\]
Letting $\delta\downarrow0$ proves metric additivity.
Carath\'eodory's theorem therefore implies that every Borel set is
$\mathcal P^s_\theta$-measurable. Lemma~\ref{lem:elementary-local-estimates}
applied to each labelled carrier gives \eqref{eq:exponent-comparison}.
If $t>s$ and $\mathcal P^{s}_{\theta}(F)<\infty$, then
\[
 \mathcal P^{t}_{\theta}(F)
 \le\lim_{\delta\downarrow0}(2\delta)^{t-s}\mathcal P^{s}_{\theta}(F)=0.
\]
\end{proof}

\begin{corollary}
\label{coro:critical-exponent}
For every $F\subset\R^d$ and $0\le\theta\le1$,
\[
 \mathcal P^{s}_{\theta}(F)=
 \begin{cases}
 \infty,&0\le s<\dim_\theta^{\mathrm P} F,\\
 0,&s>\dim_\theta^{\mathrm P} F.
 \end{cases}
\]
\end{corollary}

\begin{proof}
By Proposition~\ref{prop:metric-outer-measure},
\[
 \mathcal P^{s}_{\theta}(F)<\infty
 \quad\Longrightarrow\quad
 \mathcal P^{t}_{\theta}(F)=0\text{ for all }t>s
 \quad\Longrightarrow\quad \dim_\theta^{\mathrm P}F\le s.
\]
Hence $\mathcal P^{s}_{\theta}(F)=\infty$ for $s<\dim_\theta^{\mathrm P}F$.
For $s>\dim_\theta^{\mathrm P}F$, the definition of the critical exponent
and \eqref{eq:exponent-comparison} give $\mathcal P^{s}_{\theta}(F)=0$.
\end{proof}

Replacing each $(E_j,R_j)$ by $(E_j\cap F,R_j)$ and discarding empty
intersections does not increase the cost. Thus covers of $F$ may be
restricted to nonempty carriers contained in $F$.

\begin{theorem}
\label{thm:fundamental-properties}
\label{thm:endpoint-identities}
For $F\subset\mathbb R^d$ and $s>0$,
\begin{equation}\label{eq:hausdorff-endpoint-measure}
 \mathcal P^{s}_{0}(F)=2^s\mathcal H^s(F),
\end{equation}
where Hausdorff measure uses the cost $(\diam E)^s$.  If $\mathcal P^s$
denotes classical packing measure with the cost $(2r)^s$, there is a
constant $M_d<\infty$ such that
\begin{equation}\label{eq:packing-endpoint-measure}
 \mathcal P^s(F)\le\mathcal P^{s}_{1}(F)\le M_d\mathcal P^s(F).
\end{equation}
In particular,
\[
 \dim_0^{\mathrm P}F=\dimH F,
 \qquad \dim_1^{\mathrm P}F=\dimP F.
\]
\end{theorem}

\begin{proof}

At $\theta=0$, all admissible radii equal $R$ and $\diam E\le R$
allows at most one centre, so
\[
 \Phi^s_0(E,R)=(2R)^s\qquad(E\ne\varnothing).
\]
Assigning $R=\diam E$ to nonsingleton carriers and positive labels
with arbitrarily small total $s$-cost to singleton carriers gives
$\mathcal P^s_{0,\delta}=2^s\mathcal H^s_\delta$, hence
\eqref{eq:hausdorff-endpoint-measure} on letting $\delta\downarrow0$.

For the packing endpoint, write
\[
 Q_r^s(A)=\sup\left\{\sum_{i=1}^N(2r_i)^s:
 x_i\in A,\quad 0<r_i\le r,\quad
 |x_i-x_j|\ge r_i+r_j\ (i\ne j)\right\},
\]
and set $Q_0^s(A)=\lim_{r\downarrow0}Q_r^s(A)$. Classical packing measure
satisfies
\[
 \mathcal P^s(F)=
 \inf\left\{\sum_jQ_0^s(A_j):F\subset\bigcup_jA_j\right\}.
\]
Finite configurations suffice because every nonnegative packing sum is
the supremum of its finite subsums.

For a finite configuration $(x_i,r_i)_{i=1}^N$ admissible for
$Q_r^s(\overline A)$ and $c>1$, choose $y_i\in A$ with
$|x_i-y_i|<(1-c^{-1})\min_jr_j$. Then, for $i\ne j$,
\[
 |y_i-y_j|\ge |x_i-x_j|-|x_i-y_i|-|x_j-y_j|
 \ge c^{-1}(r_i+r_j),
\]
and $Q_r^s(A)\ge c^{-s}\sum_i(2r_i)^s$. Letting $c\downarrow1$ and
taking suprema gives $Q_r^s(\overline A)\le Q_r^s(A)$; the reverse
inequality follows from $A\subset\overline A$. Consequently,
\[
 \Phi^s_1(E,R)=Q_R^s(E)\qquad(\diam E\le R).
\]
For every labelled cover of $F$,
\[
 \mathcal P^s(F)\le\sum_jQ_0^s(E_j)
 \le\sum_j\Phi^s_1(E_j,R_j),
\]
and taking infima gives $\mathcal P^s(F)\le\mathcal P^{s}_{1}(F)$.

Fix $0<r<1$ and partition $\mathbb R^d$ into half-open
cubes of side $r/\sqrt d$.  Choose an integer $m>2\sqrt d+1$ and colour
the cubes according to their coordinate indices modulo $m$; there are
$M_d=m^d$ colours.  The closures of two distinct cubes of the same colour
are at distance greater than $2r$. For finitely many cubes $Q$ of one
colour, choose finite packings centred in $\overline{A\cap Q}$, with all
radii at most $r$.  The union of the packings is therefore a finite
packing centred in $\overline A$.
Taking suprema over finite configurations and finite collections of cubes yields
\[
 \sum_Q \Phi^s_1(A\cap Q,r)\le M_dQ_r^s(\overline A)
 =M_dQ_r^s(A).
\]
The nonempty sets $A\cap Q$, each with label $r$, form an admissible
cover.  Hence, for any fixed $\delta>0$ and $r\le\delta$,
\[
 \mathcal P^{s}_{1,\delta}(A)\le M_dQ_r^s(A).
\]
Letting $r\downarrow0$ and then $\delta\downarrow0$ gives
$\mathcal P^{s}_{1}(A)\le M_dQ_0^s(A)$. For every countable cover
$F\subset\bigcup_jA_j$,
\[
 \mathcal P^{s}_{1}(F)
 \le\sum_j\mathcal P^{s}_{1}(A_j)\le M_d\sum_jQ_0^s(A_j).
\]
Taking infima proves $\mathcal P^{s}_{1}(F)\le M_d\mathcal P^s(F)$.
Taking critical exponents in
\eqref{eq:hausdorff-endpoint-measure} and
\eqref{eq:packing-endpoint-measure} gives the dimension identities.
\end{proof}

At exponent $s=0$, $\mathcal P^0_\theta$ is counting measure for every
parameter $\theta\in[0,1]$: singleton covers give the upper bound,
while sufficiently small carriers separate any finite set of distinct
points and give the lower bound.

\subsection{Continuity}
\label{sec:analytic-properties}

Theorem~\ref{thm:quantitative-continuity} quantifies the dependence of
the packing spectrum on $\theta$. Its differential form gives the
necessary condition in Theorem~\ref{thm:attain}.

\begin{remark}
\label{rem:parameter-monotonicity}
For each $F\subset\R^d$, the function
$\theta\mapsto\dim_\theta^{\mathrm P}F$ is nondecreasing
since monotonicity of $\Phi^s_\theta$ in $\theta$ passes to
$\mathcal P^s_\theta$ and its critical exponent.
\end{remark}

\begin{lemma}
\label{lem:local_comp}
Let $0\le\theta_1<\theta_2<1$, write $p_i=(1-\theta_i)^{-1}$, and
suppose
\[
 0\le s<t<d,\qquad \frac{p_1}{p_2}<\alpha<1.
\]
If $0<R<1$ and $\diam E\le R$, then
\begin{equation}\label{eq:local_bound}
\begin{split}
 \Phi^t_{\theta_2}(E,R^\alpha)
 \le C \Phi^s_{\theta_1}(E,R)
 \bigl[&R^{\alpha t-s}+R^{t-s}
       +R^{p_1(d-s)+\alpha p_2(t-d)}\bigr],
\end{split}
\end{equation}
where $C$ depends only on $d,s,t$.
\end{lemma}

\begin{proof}
We may assume that $E$ is nonempty. Set
\[
 T=R^\alpha,\qquad \rho=R^{p_1},\qquad
 a=T^{p_2}=R^{\alpha p_2}<\rho.
\]
Let $\{(x_i,r_i)\}_{i=1}^N$ be admissible for
$\Phi^t_{\theta_2}(E,T)$. If $r_i>R$ for some $i$, then $N=1$ and
\[
 (2r_1)^t\le(2T)^t
 \le 2^{t-s}R^{\alpha t-s}\Phi^s_{\theta_1}(E,R).
\]
Otherwise, partition the indices into
\[
 I_1=\{i:\rho\le r_i\le R\},\qquad
 I_2=\{i:a\le r_i<\rho\}.
\]
The definition of $\Phi^s_{\theta_1}(E,R)$ gives
\[
 \sum_{i\in I_1}(2r_i)^t
 \le(2R)^{t-s}\Phi^s_{\theta_1}(E,R).
\]

Choose a maximal $2\rho$-separated set
$\{y_j\}_{j=1}^M\subset\overline E$, which is finite by compactness.
The configuration with centres $y_j$ and radii $\rho$ is admissible
for $\Phi^s_{\theta_1}(E,R)$, so
\begin{equation}\label{eq:maximal-rho-cardinality}
 M(2\rho)^s\le \Phi^s_{\theta_1}(E,R).
\end{equation}
Maximality gives
$\overline E\subset\bigcup_{j=1}^M B(y_j,2\rho)$.
Partition $I_2$ into sets $I_{2,j}$ such that
$x_i\in B(y_j,2\rho)$ for $i\in I_{2,j}$.
For each $j$, the balls $B(x_i,r_i)$ with $i\in I_{2,j}$ are pairwise
disjoint and contained in $B(y_j,3\rho)$. Comparing Euclidean volumes
and using $r_i\ge a$ and $t<d$ gives
\[
 \sum_{i\in I_{2,j}}(2r_i)^d\le C_d\rho^d,
 \qquad
 \sum_{i\in I_{2,j}}(2r_i)^t
 \le C_{d,t}\rho^d a^{t-d}.
\]
Summing over $j$ and applying \eqref{eq:maximal-rho-cardinality} yields
\[
 \sum_{i\in I_2}(2r_i)^t
 \le C\Phi^s_{\theta_1}(E,R)\rho^{d-s}a^{t-d}
 =C\Phi^s_{\theta_1}(E,R)
     R^{p_1(d-s)+\alpha p_2(t-d)}.
\]
Adding these estimates and taking the supremum proves
\eqref{eq:local_bound}.
\end{proof}

\begin{theorem}
\label{thm:quantitative-continuity}
Let $F\subset\mathbb R^d$, let $0\le\theta_1<\theta_2<1$, and set
$c=\frac{1-\theta_2}{1-\theta_1}$.
Then
\begin{equation}\label{eq:dim_inequality}
 \dim_{\theta_1}^{\mathrm P}F\le \dim_{\theta_2}^{\mathrm P}F\le\frac{d \dim_{\theta_1}^{\mathrm P}F}{c d+(1-c)\dim_{\theta_1}^{\mathrm P}F}.
\end{equation}
\end{theorem}

\begin{proof}
The lower bound follows from Remark~\ref{rem:parameter-monotonicity}. If
$\dim_{\theta_1}^{\mathrm P}F=d$, the upper bound follows from
$\dim_{\theta_2}^{\mathrm P}F\le d$. Otherwise, choose
$\dim_{\theta_1}^{\mathrm P}F<s<d$, write $p_i=(1-\theta_i)^{-1}$,
and set
\[
 t_0=\frac{d p_2s}{p_1d+(p_2-p_1)s}.
\]
Then $s<t_0<d$.  Fix $t_0<t<d$ and choose $\alpha$ so that
\begin{equation}\label{eq:alpha-choice}
 \max\left\{\frac{p_1}{p_2},\frac{s}{t}\right\}
 <\alpha<
 \min\left\{1,\frac{p_1(d-s)}{p_2(d-t)}\right\}.
\end{equation}
Set
\[
 \beta=\min\{\alpha t-s,\ t-s,
 p_1(d-s)+\alpha p_2(t-d)\}>0.
\]

Fix $0<\delta_0<1$ and choose $0<\delta<1$ with
$\delta^\alpha\le\delta_0$. Since $\mathcal P^s_{\theta_1}(F)=0$,
for each $\varepsilon>0$ there is a labelled cover
$\{(E_j,R_j)\}_j$ of $F$ with $R_j\le\delta$ and
\[
 \sum_j \Phi^s_{\theta_1}(E_j,R_j)<\varepsilon.
\]
Replace each $R_j$ by $T_j=R_j^\alpha$.
Since $\diam E_j\le T_j\le\delta_0$,
Lemma~\ref{lem:local_comp} gives
\[
 \sum_j\Phi^t_{\theta_2}(E_j,T_j)
 \le 3C\delta^\beta
       \sum_j\Phi^s_{\theta_1}(E_j,R_j)
 <3C\delta^\beta\varepsilon.
\]
Letting $\varepsilon\downarrow0$ and then $\delta_0\downarrow0$
gives $\mathcal P^t_{\theta_2}(F)=0$, hence
$\dim_{\theta_2}^{\mathrm P}F\le t$.
Let $t\downarrow t_0$, then $s\downarrow\dim_{\theta_1}^{\mathrm P}F$.
\end{proof}

The profile $f_{\mathrm{fast}}$ in Example~\ref{ex:extreme-profiles}
attains the upper bound in Theorem~\ref{thm:quantitative-continuity}
simultaneously for every $0\le\theta_1<\theta_2<1$.

\begin{corollary}
\label{cor:packing-spectrum-regularity}
\label{cor:continuity-at-zero}
\label{cor:zero-H-collapse}
\label{prop:regularity}
For every $F\subset\R^d$, its packing spectrum
$f(\theta)=\dim_\theta^{\mathrm P}F$ is locally Lipschitz on $[0,1)$ and
satisfies
\begin{equation}
\label{eq:dini-bound}
 0\le D^+f(\theta)
 \le\frac{f(\theta)(d-f(\theta))}{d(1-\theta)}
 \qquad(0\le\theta<1),
\end{equation}
where $D^+f(\theta)=\limsup_{h\downarrow0}
(f(\theta+h)-f(\theta))/h$ is the upper right Dini derivative.
Moreover, for each $0\le\theta<1$,
\[
f(\theta)=0\quad\Longleftrightarrow\quad\dimH F=0.
\]
\end{corollary}

\begin{proof}
Subtracting $f(\theta_1)$ from the upper bound in
Theorem~\ref{thm:quantitative-continuity} shows that the nonnegative
increment $f(\theta_2)-f(\theta_1)$ is bounded by a constant multiple of
$\theta_2-\theta_1$ on each compact subinterval of $[0,1)$.
Dividing the same increment bound by $\theta_2-\theta_1$ and letting
$\theta_2\downarrow\theta_1$ gives \eqref{eq:dini-bound}.
Setting $\theta_1=0$ in that theorem and using $f(0)=\dimH F$
proves the equivalence.
\end{proof}

\subsection{Proof of Proposition~\ref{prop:basic-properties}}
\label{sec:proof-basic-properties}
\label{sec:measurability}

\begin{proof}[Proof of Proposition~\ref{prop:basic-properties}\ref{item:basic-monotonicity} and \ref{item:basic-bounds}]
Inclusion monotonicity of Proposition~\ref{prop:metric-outer-measure}
passes to the critical exponent.
The dimension bounds follow from Theorem~\ref{thm:endpoint-identities}
and Remark~\ref{rem:parameter-monotonicity}.
\end{proof}

\begin{proof}[Proof of Proposition~\ref{prop:basic-properties}\ref{item:basic-mappings}]
We first prove
\[
 \mathcal P^{s}_{\theta}(f(F))
 \leq2^{(1-\alpha)s}L^s\mathcal P^{\alpha s}_{\theta}(F).
\]
The map $f$ extends uniquely to $\overline f:\overline F\to\R^m$
with the same H\"older constant and exponent.
For a labelled carrier $(E,R)$ with $LR^\alpha<1$, put $A=E\cap F$.
Then $\diam f(A)\leq LR^\alpha$, and compactness of
$\overline A$ gives $\overline{f(A)}=\overline f(\overline A)$.
For any finite admissible packing $(y_i,\rho_i)$ for
$\Phi^s_\theta(f(A),LR^\alpha)$, choose preimages
$x_i\in\overline A$ and set $r_i=(\rho_i/L)^{1/\alpha}$.
Since $1/\alpha\geq1$,
\[
 \begin{aligned}
 |x_i-x_j|&\geq L^{-1/\alpha}|y_i-y_j|^{1/\alpha}\\
 &\geq L^{-1/\alpha}(\rho_i+\rho_j)^{1/\alpha}\\
 &\geq r_i+r_j,
 \end{aligned}
\]
and $r_i\leq R$.
For $\theta<1$, put $p=(1-\theta)^{-1}$. Then
$r_i\geq L^{(p-1)/\alpha}R^p\geq R^p$; for $\theta=1$, the lower bound is
$r_i>0$. Hence
\begin{equation}
 \label{eq:lipschitz-local-comparison}
 \Phi^s_\theta(f(E\cap F),LR^\alpha)
 \leq2^{(1-\alpha)s}L^s\Phi^{\alpha s}_\theta(E,R).
\end{equation}
Applying \eqref{eq:lipschitz-local-comparison} to labelled covers gives
\[
 \mathcal P^{s}_{\theta,L\delta^\alpha}(f(F))
 \leq2^{(1-\alpha)s}L^s\mathcal P^{\alpha s}_{\theta,\delta}(F)
 \qquad(0<\delta<L^{-1/\alpha}).
\]
Passing to the limit $\delta\downarrow0$ and taking critical exponents
give the H\"older inequality.

The Lipschitz case is $\alpha=1$. If $f$ is bi-Lipschitz, applying
that case to $f$ and $f^{-1}$ gives equality.
\end{proof}

\begin{proof}[Proof of Proposition~\ref{prop:basic-properties}\ref{item:basic-countable}]
Let $s>\sup_j\dim_\theta^{\mathrm P}F_j$.
By Corollary~\ref{coro:critical-exponent},
$\mathcal P^s_\theta(F_j)=0$ for every $j$.
Countable subadditivity from Proposition~\ref{prop:metric-outer-measure} gives
\[
 \mathcal P^{s}_{\theta}\Bigl(\bigcup_jF_j\Bigr)
 \le\sum_j\mathcal P^{s}_{\theta}(F_j)=0.
\]
Thus $\dim_\theta^{\mathrm P}(\bigcup_jF_j)\le s$.
Letting $s\downarrow\sup_j\dim_\theta^{\mathrm P}F_j$ proves one inequality.
The reverse follows from $F_j\subset\bigcup_iF_i$.
\end{proof}

In the next two lemmas and the definition of strict labelled content, assume
$s>0$ and $0\le\theta<1$. We equip $[0,\infty]$ with the compact
topology induced by $t\mapsto t/(1+t)$, with $\infty$ mapped to $1$.

\begin{lemma}
\label{lem:local-cost-usc}
Suppose that $K_j,K\in\mathcal K(\mathbb R^d)$,
$d_H(K_j,K)\to0$, $R_j\to R\in(0,1)$, and $\diam K_j\le R_j<1$.
Then $\diam K\le R$ and
\[
 \limsup_{j\to\infty}\Phi^s_\theta(K_j,R_j)
 \le \Phi^s_\theta(K,R).
\]
\end{lemma}

\begin{proof}
Hausdorff convergence gives $\diam K=\lim_j\diam K_j\le R$.
Put $p=(1-\theta)^{-1}$ and $a=R^p/2>0$. For some $M>0$ and all
sufficiently large $j$,
\[
 K_j\subset B(0,M),\qquad R_j^p\ge a.
\]
For any admissible configuration $(x_{j,i},r_{j,i})_{i=1}^{N_j}$,
the balls $B(x_{j,i},a)$ are pairwise disjoint and contained in
$B(0,M+a)$. Comparing volumes gives
\[
 N_j a^d\le(M+a)^d,\qquad
 \sum_{i=1}^{N_j}(2r_{j,i})^s
 \le2^sN_j\le2^s\left(1+\frac{M}{a}\right)^d.
\]
Pass to a subsequence along which $\Phi^s_\theta(K_j,R_j)$ converges to
$\limsup_j\Phi^s_\theta(K_j,R_j)$, and choose admissible configurations
$(x_{j,i},r_{j,i})_{i=1}^{N_j}$ with
\[
 \Phi^s_\theta(K_j,R_j)-\frac{1}{j}
 \le\sum_{i=1}^{N_j}(2r_{j,i})^s
 \le\Phi^s_\theta(K_j,R_j).
\]
Passing to a further subsequence, we may assume that $N_j=N$,
$x_{j,i}\to x_i$, and $r_{j,i}\to r_i$ for $1\le i\le N$.
Then
\[
 x_i\in K,\qquad R^p\le r_i\le R,\qquad
 |x_i-x_\ell|\ge r_i+r_\ell\quad(i\ne\ell).
\]
Consequently,
\[
 \limsup_{j\to\infty}\Phi^s_\theta(K_j,R_j)
 =\sum_{i=1}^N(2r_i)^s\le\Phi^s_\theta(K,R).
\]
\end{proof}

For $0<\delta<1$, define the strict labelled content by
\[
 \mathcal P^{s}_{\theta,<\delta}(F)
 :=\inf\left\{\sum_j\Phi^s_\theta(E_j,R_j):
 F\subset\bigcup_jE_j,\quad \diam E_j\le R_j<\delta,\quad R_j>0\right\}.
\]
For $0<\delta'<\delta$,
\begin{equation}\label{eq:strict-content-squeeze}
 \mathcal P^{s}_{\theta,\delta}(F)
 \le\mathcal P^{s}_{\theta,<\delta}(F)
 \le\mathcal P^{s}_{\theta,\delta'}(F).
\end{equation}
Hence
\[
 \lim_{\delta\downarrow0}\mathcal P^{s}_{\theta,<\delta}(F)
 =\mathcal P^{s}_{\theta}(F).
\]

\begin{lemma}
\label{lem:open-labelled-carriers}
Restricting the carriers in $\mathcal P^{s}_{\theta,<\delta}$ to open
sets preserves the infimum.
The map
$K\mapsto\mathcal P^{s}_{\theta,<\delta}(K)$ is upper semicontinuous
on $\mathcal K(\mathbb R^d)$.
\label{lem:meas-USC-open-content}
\end{lemma}

\begin{proof}
Let $(E,R)$ be a nonempty carrier with $\diam E\le R<\delta$.
The closure $K=\overline E$ is compact, and
$\Phi^s_\theta(K,R)=\Phi^s_\theta(E,R)$.
Choose $R_m\downarrow R$ with $R<R_m<\delta$, and
$\varepsilon_m\downarrow0$ with $0<2\varepsilon_m\le R_m-R$.
The open sets
\[
 U_m=\{x:\dist(x,K)<\varepsilon_m\}
\]
satisfy $\diam U_m\le R_m$ and $d_H(\overline U_m,K)\to0$.
Lemma~\ref{lem:local-cost-usc} gives
\[
 \limsup_{m\to\infty}\Phi^s_\theta(U_m,R_m)
 \le \Phi^s_\theta(E,R).
\]
For $\varepsilon>0$, enlarge each nonempty member $(E_j,R_j)$ of a
countable cover of finite cost to an open labelled carrier $(U_j,S_j)$
satisfying
\[
 E_j\subset U_j,\qquad \diam U_j\le S_j<\delta,\qquad
 \Phi^s_\theta(U_j,S_j)\le \Phi^s_\theta(E_j,R_j)+\varepsilon2^{-j}.
\]
The total cost of $(U_j,S_j)_j$ is at most
$\sum_j\Phi^s_\theta(E_j,R_j)+\varepsilon$.
If $\mathcal P^{s}_{\theta,<\delta}(F)<\infty$, taking infima over
admissible labelled covers of $F$ gives
\[
 \mathcal P^{s}_{\theta,<\delta}(F)
 \le\inf_{\substack{(U_j,S_j)_j\text{ admissible}\\ U_j\text{ open}}}
       \sum_j\Phi^s_\theta(U_j,S_j)
 \le\mathcal P^{s}_{\theta,<\delta}(F)+\varepsilon.
\]
Letting $\varepsilon\downarrow0$ proves equality of the infima over
arbitrary and open carriers.
If $\mathcal P^{s}_{\theta,<\delta}(F)=\infty$, every admissible
open cover has infinite cost, so equality also holds.

For compact $K$, every open cover has a finite subcover of no greater
cost. If $\mathcal P^{s}_{\theta,<\delta}(K)<a$, choose an open
labelled cover $(U_j,S_j)_{j=1}^N$ with total cost below $a$.
Since $K\subset\bigcup_{j=1}^N U_j$, there exists $\rho>0$ such that,
for every $K'\in\mathcal K(\mathbb R^d)$,
\[
 d_H(K',K)<\rho
 \quad\Longrightarrow\quad
 K'\subset\bigcup_{j=1}^N U_j
 \quad\Longrightarrow\quad
 \mathcal P^{s}_{\theta,<\delta}(K')
 \le\sum_{j=1}^N\Phi^s_\theta(U_j,S_j)<a.
\]
Thus $K\mapsto\mathcal P^{s}_{\theta,<\delta}(K)$ is upper
semicontinuous.
\end{proof}

\begin{proof}[Proof of Proposition~\ref{prop:basic-properties}\ref{item:basic-baire}]
\quad

\noindent \textbf{Claim.}
For $s>0$, the map
$K\mapsto\mathcal P^{s}_{\theta}(K)$ on $\mathcal K(\mathbb R^d)$
is of Baire class at most $2$. Moreover,
\[
 \{K:\mathcal P^{s}_{\theta}(K)=0\}\text{ is }G_\delta,
 \qquad
 \{K:\mathcal P^{s}_{\theta}(K)>0\}\text{ is }F_\sigma.
\]

\begin{proof}[Proof of the claim]
Set $A_k(K)=\mathcal P^{s}_{\theta,<2^{-k}}(K)$ for $k\ge1$.
By Lemma~\ref{lem:open-labelled-carriers}, each $A_k$ is upper
semicontinuous, and \eqref{eq:strict-content-squeeze} gives
\[
 A_k(K)\uparrow\mathcal P^{s}_{\theta}(K).
\]
With the chosen topology on $[0,\infty]$, each $A_k$ is of Baire
class at most $1$. Their pointwise limit
$K\mapsto\mathcal P^{s}_{\theta}(K)$ is therefore of Baire class at
most $2$.
Since $A_k\ge0$,
\begin{align*}
 \{\mathcal P^{s}_{\theta}=0\}
 &=\bigcap_{k,\ell\ge1}\{A_k<1/\ell\},\\
 \{\mathcal P^{s}_{\theta}>0\}
 &=\bigcup_{k,\ell\ge1}\{A_k\ge1/\ell\}.
\end{align*}
Upper semicontinuity of $A_k$ makes $\{A_k<1/\ell\}$ open and
$\{A_k\ge1/\ell\}$ closed.
\end{proof}

The critical-exponent property gives, for $0\le\alpha<d$,
\[
 \{K:\dim_\theta^{\mathrm P}K>\alpha\}
 =\bigcup_{q\in\mathbb Q\cap(\alpha,d)}
   \{K:\mathcal P^{q}_{\theta}(K)>0\},
\]
and, for $0<\alpha\le d$,
\[
 \{K:\dim_\theta^{\mathrm P}K<\alpha\}
 =\bigcup_{q\in\mathbb Q\cap(0,\alpha)}
   \{K:\mathcal P^{q}_{\theta}(K)=0\}.
\]
By the claim and
$0\le\dim_\theta^{\mathrm P}K\le d$,
\[
 \{K:\dim_\theta^{\mathrm P}K>\alpha\}\text{ is }F_\sigma,
 \qquad
 \{K:\dim_\theta^{\mathrm P}K<\alpha\}\text{ is }G_{\delta\sigma}
 \quad(\alpha\in\mathbb R).
\]
For every open interval $I\subset\R$, the preimage
$\{K:\dim_\theta^{\mathrm P}K\in I\}$ is $G_{\delta\sigma}$.
Since $\mathcal K(\R^d)$ is Polish, the Lebesgue--Hausdorff
characterization~\cite{Kechris1995} gives Baire class at most $2$.
\end{proof}

\begin{lemma}
\label{lem:add-label-perturbation}
Fix $0\le\theta<1$, $t>0$, and put $p=(1-\theta)^{-1}$.
Suppose that nonempty compact sets $C_j$ converge in the Hausdorff
metric to $C$, that $0<R_j<1$, $R_j\to R>0$, and
$\diam C_j\le R_j$.
If $R<S<1$, there is a bounded open neighbourhood $U$ of $C$, with
$\diam U\le S$, such that
\[
 \Phi^t_\theta(U,S)
 \le (S/R)^t\liminf_{j\to\infty}\Phi^t_\theta(C_j,R_j).
\]
\end{lemma}

\begin{proof}
Choose $\rho>0$ sufficiently small that
\[
 2\rho<S-R,
 \qquad 2\rho<(1-R/S)S^p,
\]
and let $U=\{x:\dist(x,C)<\rho\}$.  For sufficiently large $j$,
$R_j<S$ and
\[
 \rho+d_H(C_j,C)<(1-R_j/S)S^p.
\]
Given an admissible packing $(x_i,r_i)_i$ for $(U,S)$, with
$x_i\in\overline U$, choose
$y_i\in C_j$ with $|x_i-y_i|\le\rho+d_H(C_j,C)$ and put
$c_j=R_j/S$.  Then
\[
 R_j^p\le c_jS^p\le c_jr_i\le R_j
\]
and, for $i\ne\ell$,
\[
 |y_i-y_\ell|
 \ge r_i+r_\ell-2\bigl(\rho+d_H(C_j,C)\bigr)
 \ge c_j(r_i+r_\ell).
\]
Thus $(y_i,c_jr_i)_i$ is admissible for $(C_j,R_j)$, so
$\Phi^t_\theta(U,S)\le c_j^{-t}\Phi^t_\theta(C_j,R_j)$.
Taking the lower limit proves the assertion.
\end{proof}

\begin{lemma}
\label{lem:add-increasing-sets}
Fix $0\le\theta<1$ and $0<s<u$.  Let $A_j$ be an increasing
sequence of subsets of a common bounded set, and put $A=\bigcup_j A_j$.
For $0<\delta<\delta'<1$,
\begin{equation}
\label{eq:add-increasing-sets}
 \mathcal P^u_{\theta,\delta'}(A)
 \le (2\delta)^{u-s}
       \sup_j\mathcal P^s_{\theta,\delta}(A_j).
\end{equation}
\end{lemma}

\begin{proof}
We may assume that $A\ne\varnothing$ and choose
\[
 \sup_j\mathcal P^s_{\theta,\delta}(A_j)<M<\infty.
\]
Put $p=(1-\theta)^{-1}$ and $X=\overline A$. Choose labelled
$\delta$-covers $(C_{j,i},R_{j,i})_{i\ge1}$ of $A_j$ with
nonempty compact carriers $C_{j,i}\subset X$ and
\[
 \sum_i\Phi^s_\theta(C_{j,i},R_{j,i})<M.
\]
Here we intersect carriers with $X$ and take closures, adding
singleton carriers within the cost bound if the cover is finite.
Arrange the labels in nonincreasing order.
Since
$i(2R_{j,i})^s\le\sum_{k\le i}\Phi^s_\theta(C_{j,k},R_{j,k})<M$,
\begin{equation}\label{eq:add-label-tail}
 R_{j,i}\le a_i:=\tfrac12(M/i)^{1/s}\longrightarrow0.
\end{equation}

Along a diagonal subsequence in $j$, compactness gives
\[
 C_{j,i}\longrightarrow C_i,\qquad R_{j,i}\longrightarrow R_i
 \qquad(i\ge1),
\]
where $\diam C_i\le R_i\le\delta$. Set
$L_i=\liminf_j\Phi^u_\theta(C_{j,i},R_{j,i})$ along this subsequence.
Fatou's lemma and the exponent comparison give
\[
 \sum_iL_i
 \le\liminf_j\sum_i\Phi^u_\theta(C_{j,i},R_{j,i})
 \le(2\delta)^{u-s}M.
\]

Fix $1<\lambda<\delta'/\delta$ and $\varepsilon>0$.
If $R_i>0$, Lemma~\ref{lem:add-label-perturbation}, applied at
exponent $u$, gives an open neighbourhood $U_i$ of $C_i$ with
label $S_i=\lambda R_i$ and
$\Phi^u_\theta(U_i,S_i)\le\lambda^uL_i$.
If $R_i=0$, write $C_i=\{x_i\}$ and choose $0<S_i<\delta'$ so that
\[
 U_i=B(x_i,S_i^p/4),\qquad
 \Phi^u_\theta(U_i,S_i)=(2S_i)^u<\varepsilon2^{-i}.
\]
Thus $U=\bigcup_iU_i$ satisfies
\begin{equation}\label{eq:add-limit-head-cost}
 \mathcal P^u_{\theta,\delta'}(U)
 \le\sum_i\Phi^u_\theta(U_i,S_i)
 \le\lambda^u(2\delta)^{u-s}M+\varepsilon.
\end{equation}

Fix $m$ and take $N$ large enough that $a_{N+1}\le\delta'$.
Hausdorff convergence allows us to choose $j\ge m$ from the
subsequence with $C_{j,i}\subset U$ for $i\le N$. Since
$A_m\subset A_j$,
\[
 A_m\setminus U\subset\bigcup_{i>N}C_{j,i},
 \qquad R_{j,i}\le a_{N+1}\quad(i>N).
\]
Consequently,
\[
 \mathcal P^u_{\theta,\delta'}(A_m\setminus U)
 \le\sum_{i>N}\Phi^u_\theta(C_{j,i},R_{j,i})
 \le(2a_{N+1})^{u-s}M\longrightarrow0
 \qquad(N\to\infty).
\]
Countable subadditivity now gives
\[
 \mathcal P^u_{\theta,\delta'}(A\setminus U)=0,
 \qquad
 \mathcal P^u_{\theta,\delta'}(A)
 \le\lambda^u(2\delta)^{u-s}M+\varepsilon.
\]
Let $\varepsilon\downarrow0$, $\lambda\downarrow1$, and
$M\downarrow\sup_j\mathcal P^s_{\theta,\delta}(A_j)$.
\end{proof}

\begin{proof}[Proof of Proposition~\ref{prop:basic-properties}\ref{item:basic-analytic}]
Monotonicity gives $\dim_\theta^{\mathrm P}K\le\dim_\theta^{\mathrm P}A$
for every compact $K\subset A$. At $\theta=1$, the supremum identity follows from the classical compact
approximation of packing dimension~\cite[Lemma~2.3]{Xiao1996}.
Fix $0\le\theta<1$. By countable stability, it suffices to consider
bounded analytic $A$.

Fix $0<s<t_0<\dim_\theta^{\mathrm P}A$ and choose
$0<\delta_0<1/4$ and $c>0$ with
$\mathcal P^{t_0}_{\theta,\delta_0}(A)>c$.
Let $\pi:\N^{\N}\to A$ be a continuous surjection, and put
\[
 t_j=s+(t_0-s)2^{-j},\qquad
 \delta_j=\tfrac12\delta_0(1+2^{-j}).
\]
Starting with $A_0=A$, suppose that $b_1,\ldots,b_{j-1}$ have been
chosen. For $b\in\N$, set
\[
 B_b=\pi\{\alpha\in\N^{\N}:\alpha_i\le b_i\ (i<j),
                              \quad\alpha_j\le b\}.
\]
Since $B_b\uparrow A_{j-1}$, Lemma~\ref{lem:add-increasing-sets} gives
\[
 c<\mathcal P^{t_{j-1}}_{\theta,\delta_{j-1}}(A_{j-1})
 \le(2\delta_j)^{t_{j-1}-t_j}
       \sup_b\mathcal P^{t_j}_{\theta,\delta_j}(B_b).
\]
As $2\delta_j<1$, choose $b_j$ and set $A_j=B_{b_j}$ so that
$\mathcal P^{t_j}_{\theta,\delta_j}(A_j)>c$.

The sets $C_j=\overline{A_j}$ are nonempty nested compact sets.
They converge in the Hausdorff metric to $K=\bigcap_jC_j$, and
\[
 K=\pi\Bigl(\prod_{j\ge1}\{1,\ldots,b_j\}\Bigr)\subset A.
\]
For the inclusion from left to right, take $x\in K$ and choose
$\alpha^{(j)}\in\N^{\N}$ with
\[
 \alpha_i^{(j)}\le b_i\quad(i\le j),\qquad
 |\pi(\alpha^{(j)})-x|<1/j.
\]
A diagonal subsequence converges coordinatewise to
$\alpha\in\prod_i\{1,\ldots,b_i\}$. Continuity of $\pi$ gives
$x=\pi(\alpha)$. The reverse inclusion follows from the definition
of $A_j$.

For $R<\delta_0/2$, we have $R<\delta_j$, $2R<1$, and
$\Phi^{t_j}_\theta(E,R)\le\Phi^s_\theta(E,R)$. Hence
\[
 \mathcal P^s_{\theta,<\delta_0/2}(C_j)
 \ge\mathcal P^{t_j}_{\theta,\delta_j}(A_j)>c.
\]
Upper semicontinuity in Lemma~\ref{lem:open-labelled-carriers} gives
\[
 \mathcal P^s_\theta(K)
 \ge\mathcal P^s_{\theta,<\delta_0/2}(K)
 \ge\limsup_{j\to\infty}\mathcal P^s_{\theta,<\delta_0/2}(C_j)
 \ge c>0.
\]
Thus $\dim_\theta^{\mathrm P}K\ge s$. Let $s\uparrow\dim_\theta^{\mathrm P}A$;
if $\dim_\theta^{\mathrm P}A=0$, any singleton subset suffices.
\end{proof}

The proof also shows that, for every $0\le\theta<1$ and
$0<s<\dim_\theta^{\mathrm P}A$, there is a compact $K\subset A$
with $\mathcal P^s_\theta(K)>0$.

\section{Measure-theoretic characterization}
\label{sec:single-measure-variational-principle}

For $0\le\theta<1$, the packing spectrum has a variational principle
in terms of a windowed local exponent. For a nonempty compact set
$F\subset\R^d$, let $\mathcal M_1(F)$ denote the set of Borel probability
measures with support contained in $F$. We recall the classical
Hausdorff and packing variational principles
\cite{Frostman1935,Cutler1995}:
\[
 \begin{aligned}
 \dimH F&=\sup_{\mu\in\mathcal M_1(F)}
   \operatorname*{ess\,inf}_{x\sim\mu}
   \liminf_{r\downarrow0}\frac{\log\mu(B(x,r))}{\log r},\\
 \dimP F&=\sup_{\mu\in\mathcal M_1(F)}
   \operatorname*{ess\,inf}_{x\sim\mu}
   \limsup_{r\downarrow0}\frac{\log\mu(B(x,r))}{\log r}.
 \end{aligned}
\]

Unless stated otherwise, throughout this section fix
$0\le\theta<1$ and write $p=p(\theta)=1/(1-\theta)$.

\begin{definition}
\label{def:windowed-local-dimension}
For a compactly supported Borel probability measure $\mu$ and
$0\le\theta<1$, define the extended nonnegative local exponent
\begin{equation}\label{eq:density}
 d_\theta^\mu(x)
 =\liminf_{R\downarrow0}\sup_{R^p\le r\le R}
   \frac{\log\mu(B(x,r))}{\log r},
 \qquad x\in\supp\mu.
\end{equation}
\end{definition}
When $\theta=0$, the radius window in \eqref{eq:density} is the
singleton $\{R\}$, so $d_0^\mu$ is the classical lower local dimension.
We write $\supp\mu\Subset F$ when $\supp\mu$ is a
compact subset of $F$.

The classical upper local dimension need not equal
$\lim_{\theta\uparrow1}d_\theta^\mu(x)$;
Example~\ref{ex:extreme-profiles} gives strict inequality at every
point of a compact digit set.

\begin{theorem}\label{thm:main}
For every nonempty analytic set $F\subset\R^d$ and every
$0\le\theta<1$,
\begin{equation}\label{eq:main}
 \dim_\theta^{\mathrm P}F
 =\sup_{\substack{\mu\text{ probability}\\\supp\mu\Subset F}}
       \operatorname*{ess\,inf}_{x\sim\mu}d_\theta^\mu(x).
\end{equation}
\end{theorem}

In particular, if $F\subset\R^d$ is nonempty and analytic and
$0<s<\dim_\theta^{\mathrm P}F$, there exists a probability measure
$\mu$ with $\supp\mu\Subset F$ such that, for $\mu$-almost every
$x$ and every $0<v<s$, some $R_0(x,v)>0$ satisfies
\begin{equation}\label{eq:suitable-radius}
 \forall\,0<R<R_0(x,v)\quad
 \exists\,r\in[R^{1/(1-\theta)},R]:
 \qquad\mu(B(x,r))\le r^v.
\end{equation}
We use \eqref{eq:suitable-radius} in the product construction.
The supremum need not be attained; the final subsection
gives a sufficient condition, and Example~\ref{ex:add-nonattainment}
shows why a subcritical statement is necessary in general.

The intermediate-dimensional Frostman lemma permits a scale-dependent
measure controlling all radii in a window
\cite[Proposition~2.3]{FFK2020}; see also \cite{Banaji2023}.
In \eqref{eq:suitable-radius}, a fixed measure satisfies the mass
bound at some radius in every sufficiently small window.
Angelini's ball characterization instead requires subcritical mass
decay at every radius in arbitrarily small windows, almost everywhere
\cite[Theorem~4.3]{Angelini2026c}.

\subsection{Carrier domination at subcritical exponents}

\begin{lemma}
\label{lem:window-measurability}
Let $\mu$ be a compactly supported Borel probability measure and let
$0\le\theta<1$. For every $R\in(0,1)$, the maximum
\begin{equation}\label{eq:window-maximum}
 M_R^\mu(x)=\max_{R^p\le r\le R}
                 \frac{\log\mu(B(x,r))}{\log r}
\end{equation}
is attained for every $x\in\supp\mu$, and the function $M_R^\mu$
is upper semicontinuous on $\supp\mu$. The exponent
$d_\theta^\mu$ is Borel measurable. For $a\ge0$ and $0<\delta<1$,
the set
$\{x\in\supp\mu:M_R^\mu(x)\ge a\text{ for all }0<R<\delta\}$
is closed in $\supp\mu$.
\end{lemma}

\begin{proof}
The ball-mass map $(x,r)\mapsto\mu(B(x,r))$ is positive on
$\supp\mu\times(0,1)$ and lower semicontinuous by Fatou's lemma.
The quotient in \eqref{eq:density} is therefore finite and upper
semicontinuous. Compactness of $[R^p,R]$ gives attainment and upper
semicontinuity of $M_R^\mu$ in $x$. Moreover,
\[
 d_\theta^\mu(x)=\sup_{m\ge2}\inf_{0<R<1/m}M_R^\mu(x)
\]
is a countable supremum of upper semicontinuous functions, and hence
Borel measurable. The set in the last assertion is an intersection of closed
superlevel sets of $M_R^\mu$.
\end{proof}

We apply Howroyd's weighted-cover duality method~\cite{Howroyd1995}
to the costs $\Phi^t_\theta(E,R)$. For a nonempty compact set
$F\subset\R^d$, $t>0$, and $0<\delta<1/4$, define, for real-valued
$f\in C(F)$,
\begin{equation}\label{eq:weighted}
 q_\delta(f)
 =\inf\left\{\sum_i a_i\Phi^t_\theta(E_i,R_i):
 \begin{array}{l}
 f\le\sum_i a_i\mathbf1_{E_i}\text{ on }F,\\
 a_i\ge0,\quad 0<R_i\le\delta,\quad\diam E_i\le R_i
 \end{array}\right\}.
\end{equation}
The families are finite or countable, including the empty family.
Put $A_\delta=q_\delta(1)$.

\begin{lemma}
\label{lem:weighted-content-comparison}
Let $F\subset\R^d$ be nonempty and compact, let $t>0$, and let $K\ge3$
be an integer. Define $A_\delta$ by \eqref{eq:weighted} with exponent
$t$. For every $u>t$,
\begin{equation}\label{eq:rounding}
 \mathcal P^u_{\theta,2^{1-K}}(F)
 \le C_{d,u}\sup_{k\ge K}\bigl((k+1)^2 2^{-k(u-t)}\bigr)
                  A_{2^{-K}}.
\end{equation}
\end{lemma}

\begin{proof}
Put $\delta=2^{-K}$ and take a family
$(a_i,E_i,R_i)_i$ in \eqref{eq:weighted} with $f=1$ and finite
total cost. Replacing $E_i$ by $E_i\cap F$ and discarding empty
sets does not increase the cost, so assume
$\varnothing\ne E_i\subset F$. Set
\[
 I_k=\{i:2^{-k-1}<R_i\le2^{-k}\},\qquad
 \lambda_k=\frac{1}{2(k+1)^2},\qquad
 F_k=\left\{x\in F:\sum_{i\in I_k}a_i\mathbf1_{E_i}(x)
                         \ge\lambda_k\right\}.
\]
Since $\sum_{k\ge K}\lambda_k<1$, the sets $F_k$ cover $F$.
For each $k$, put $\rho=2^{-k}$ and $S=2\rho$, and partition
$\R^d$ into half-open cubes $Q$ of side $\rho/\sqrt d$.
Given an admissible packing $(x_j,r_j)_j$ for $(F_k\cap Q,S)$,
choose $y_j\in F_k\cap Q$ with $|y_j-x_j|<S^p/4$.
For $i\in I_k$, set $c_i=R_i/S\in(1/4,1/2]$. Then
\[
 R_i^p\le R_iS^{p-1}\le c_i r_j\le R_i,\qquad
 |y_j-y_\ell|\ge r_j+r_\ell-S^p/2
 \ge c_i(r_j+r_\ell)\quad(j\ne\ell).
\]
Thus $(y_j,c_i r_j)_{j:y_j\in E_i}$ is admissible for $(E_i,R_i)$,
and
\[
 \lambda_k\sum_j(2r_j)^t
 \le\sum_{i\in I_k}a_i\sum_{j:y_j\in E_i}(2r_j)^t
 \le4^t\sum_{\substack{i\in I_k\\E_i\cap Q\ne\varnothing}}
               a_i\Phi^t_\theta(E_i,R_i).
\]
Each $E_i$, having diameter at most $\rho$, meets at most $C_d$
cubes.
Since \(r_j\le S\), the preceding estimate implies, for \(u\ge t\),
\begin{equation}\label{eq:grid-local}
 \sum_Q\Phi^u_\theta(F_k\cap Q,S)
 \le C_{d,u}\lambda_k^{-1}\rho^{u-t}
       \sum_{i\in I_k}a_i\Phi^t_\theta(E_i,R_i).
\end{equation}
The carriers $F_k\cap Q$, labelled $2^{1-k}$, cover $F$.
For $u>t$, sum \eqref{eq:grid-local} over $k\ge K$, use
$\lambda_k^{-1}=2(k+1)^2$, and take the infimum over the families
in \eqref{eq:weighted} with $f=1$. This proves
\eqref{eq:rounding}.
\end{proof}

\begin{proposition}
\label{prop:carrier-frostman}
Let $F\subset\R^d$ be nonempty and compact and let $0\le\theta<1$.
If $0<t<\dim_\theta^{\mathrm P}F$, there exist
$\mu\in\mathcal M_1(F)$ and $C,\delta_0>0$ such that
\begin{equation}\label{eq:carrier-bound}
 \mu^*(E)\le C\Phi^t_\theta(E,R)
 \qquad(0<R<\delta_0,\ \diam E\le R).
\end{equation}
Here $\mu^*$ is the outer measure induced by $\mu$.
\end{proposition}
\begin{proof}

Choose $t<u<\dim_\theta^{\mathrm P}F$ and an integer $K\ge3$
sufficiently large that $\mathcal P^u_{\theta,2^{1-K}}(F)>0$.
Lemma~\ref{lem:weighted-content-comparison} gives $A_{2^{-K}}>0$.
Put $\delta=2^{-K}$.

Write $q=q_\delta$ and $A=A_\delta$. A finite labelled
cover of $F$ and Lemma~\ref{lem:elementary-local-estimates} give
$0\le q(f)\le C\|f\|_\infty$. Combining weighted covers and
rescaling their coefficients show that $q$ is sublinear; moreover,
$q(f)=0$ for $f\le0$. Hence $c\mathbf1\mapsto cA$ is dominated by
$q$ on the constants. Hahn--Banach extends it to a linear functional
$L\le q$ on $C(F)$. Since $L(f)=-L(-f)\ge0$ for $f\ge0$, the
Riesz representation theorem yields a positive regular Borel measure
$\nu$ on $F$ with $\nu(F)=A>0$. For an open labelled carrier
$(U,R)$ with $R\le\delta$, a one-member weighted cover gives
$L(f)\le \Phi^t_\theta(U,R)$ whenever
$0\le f\le\mathbf1_{U\cap F}$. Taking the supremum over these
continuous functions gives
\[
 \nu(U\cap F)\le \Phi^t_\theta(U,R).
\]
To extend this domination to arbitrary carriers, let $(E,R)$
be nonempty with $0<R<\delta$ and $\diam E\le R$.
Lemma~\ref{lem:open-labelled-carriers} provides
$R_m\downarrow R$, with $R<R_m<\delta$, and open sets
$U_m\supset E$ such that $\diam U_m\le R_m$ and
\[
 \limsup_{m\to\infty}\Phi^t_\theta(U_m,R_m)
 \le\Phi^t_\theta(E,R).
\]
Extend $\nu$ by zero outside $F$. Since $E\cap F\subset U_m\cap F$,
\[
 \nu^*(E)=\nu^*(E\cap F)
 \le\nu(U_m\cap F)
 \le \Phi^t_\theta(U_m,R_m).
\]
It follows that
\[
 \nu^*(E)
 \le\liminf_{m\to\infty}\Phi^t_\theta(U_m,R_m)
 \le \Phi^t_\theta(E,R).
\]
The same bound is immediate for $E=\varnothing$.
Finally, normalize by setting $\mu=\nu/A_\delta$.
This gives the required bound with $\delta_0=\delta$ and
$C=A_\delta^{-1}$.
\end{proof}

\subsection{Sublevel estimates and mass distribution}

\begin{proposition}
\label{thm:sublevel}
Let $\mu$ be a compactly supported Borel probability measure.
For every $a\ge0$,
\begin{equation}\label{eq:sublevel}
 \dim_\theta^{\mathrm P}
 \{x\in\supp\mu:d_\theta^\mu(x)\le a\}\le a.
\end{equation}
\end{proposition}
\begin{proof}
Put $R_k=2^{-k}$ and, for $v\ge0$, define
\begin{equation}\label{eq:bad}
 A_k(v)=\{x\in\supp\mu:
   \mu(\overline {B(x,r)})\ge r^v\text{ for all }R_k^p\le r\le R_k\}.
\end{equation}
For fixed $r>0$, $x\mapsto\mu(\overline B(x,r))$ is upper
semicontinuous, so $A_k(v)$ is a closed subset of the compact set
$\supp\mu$.
Partition $\R^d$ into half-open cubes $Q$ of side $R_k/\sqrt d$,
and use the nonempty carriers $E_Q=A_k(v)\cap Q$, labelled $R_k$.
For an admissible packing $(x_i,r_i)_i$ in $E_Q$, the open balls are
disjoint. The closed balls have multiplicity at most two: a common
point of two balls lies on both boundaries, and equality in the
triangle inequality forces the two radial directions to be opposite;
three directions cannot be pairwise opposite.
Since $\overline E_Q\subset A_k(v)$, for $u>v$,
\[
 \sum_i(2r_i)^u
 \le 2^u R_k^{u-v}\sum_i\mu(\overline {B(x_i,r_i)})
 \le 2^{u+1}R_k^{u-v}\mu(Q^+),
\]
where $Q^+=\{z:\dist(z,\overline Q)\le R_k\}$.
The sets $Q^+$ have multiplicity bounded by a constant depending
only on $d$.
Taking the supremum over admissible packings in each $E_Q$ and
summing over $Q$ gives
\begin{equation}\label{eq:bad-cost}
 \sum_Q\Phi^u_\theta(E_Q,R_k)\le C_{d,u}2^{-k(u-v)}.
\end{equation}

Fix $u>a$ and choose $a<b<v<u$. If $d_\theta^\mu(x)\le a$,
there are arbitrarily small $R$ such that
\[
 \mu(B(x,r))>r^b\qquad(R^p\le r\le R).
\]
Choose $R\le R_k<2R$. For $R_k^p\le r\le R_k$, if $r\le R$,
then $\mu(\overline {B(x,r)})>r^b\ge r^v$. If $r>R$, then,
for sufficiently small $R$,
\[
 \mu(\overline {B(x,r)})\ge\mu(B(x,R))
 >R^b\ge(2R)^v\ge r^v.
\]
Thus $\{d_\theta^\mu\le a\}\subset\limsup_k A_k(v)$.
For every $m$, the covers used in \eqref{eq:bad-cost}, with $k\ge m$,
cover this limsup and give
\[
 \mathcal P^u_{\theta,2^{-m}}\bigl(\limsup_k A_k(v)\bigr)
 \le C_{d,u}\sum_{k\ge m}2^{-k(u-v)}\longrightarrow0.
\]
Therefore $\mathcal P^u_\theta\{d_\theta^\mu\le a\}=0$.
Letting $u\downarrow a$ proves \eqref{eq:sublevel}.
\end{proof}
For every compactly supported Borel probability measure $\mu$ and $u>0$,
\[
 \{x\in\supp\mu:d_\theta^\mu(x)<u\}
 =\bigcup_{n\ge2}
   \{x\in\supp\mu:d_\theta^\mu(x)\le(1-1/n)u\}.
\]
By \eqref{eq:sublevel} and the critical-exponent property, each set on
the right has zero $\mathcal P^u_\theta$-measure. Countable subadditivity gives
\begin{equation}\label{eq:null-sublevel}
 \mathcal P^u_\theta\{x\in\supp\mu:d_\theta^\mu(x)<u\}=0.
\end{equation}

\begin{proposition}
\label{prop:mass-distribution}
Let $\mu$ be a compactly supported Borel probability measure and let
$t\ge0$. Suppose that $A_0\subset\supp\mu$ is Borel and that some
$0<R_0<1$ satisfies
\[
 \forall x\in A_0\ \forall\,0<R<R_0\ \exists r\in[R^p,R]:
 \qquad\mu(B(x,r))\le r^t.
\]
Then, for every $0<\delta<R_0$,
\begin{equation}\label{eq:uniform-mass-content}
 \mathcal P^t_{\theta,\delta}(A_0)
 \ge\frac{2^t}{B_d}\mu(A_0),
\end{equation}
where $B_d$ is the dimensional constant in the assigned-ball form
of Besicovitch's covering theorem.
Consequently, if $A\subset\supp\mu$ is Borel, $\mu(A)>0$, and
$d_\theta^\mu(x)>t$ on $A$, then
\begin{equation}\label{eq:mass-direction}
 \dim_\theta^{\mathrm P}A\ge t.
\end{equation}
\end{proposition}
\begin{proof}
For a labelled carrier $(E,R)$ with $R<R_0$ and
$A_0\cap E\ne\varnothing$, choose one of the prescribed balls at
each point of $A_0\cap E$. The assigned-ball form of Besicovitch's
covering theorem \cite[Section~3]{Loeb1997} gives at most $B_d$
subfamilies of pairwise disjoint open balls whose union covers
$A_0\cap E$. Each subfamily is finite, because its centres lie in a
bounded set and its radii are at least $R^p>0$.
Subadditivity of $\mu^*$ shows that one family satisfies
\[
 \frac{\mu^*(A_0\cap E)}{B_d}
 \le\sum_i\mu(B(x_i,r_i))
 \le\sum_i r_i^t
 \le2^{-t}\Phi^t_\theta(E,R).
\]
Summing over any labelled $\delta$-cover of $A_0$, with
$0<\delta<R_0$, and taking the infimum proves
\eqref{eq:uniform-mass-content}.

For \eqref{eq:mass-direction}, the case $t=0$ is immediate.
If $t>0$, set, for $m\ge2$,
\[
 A_m=\left\{x\in A:M_R^\mu(x)\ge t
       \text{ for every }0<R<1/m\right\}.
\]
Each $A_m$ is Borel by Lemma~\ref{lem:window-measurability}. Since $d_\theta^\mu>t$ on $A$,
these sets increase to $A$. Choose $m$ with $\mu(A_m)>0$ and set
$A_0=A_m$, $R_0=1/m$. Attainment of the maximum in
\eqref{eq:window-maximum} gives the required radius at every point
and every $0<R<R_0$. Thus \eqref{eq:uniform-mass-content} gives
$\mathcal P^t_\theta(A_0)>0$, whence
$\dim_\theta^{\mathrm P}A\ge\dim_\theta^{\mathrm P}A_0\ge t$.
\end{proof}

\begin{proof}[Proof of Theorem~\ref{thm:main}]
First suppose that $F$ is compact.
For every $0<t<\dim_\theta^{\mathrm P}F$, Proposition~\ref{prop:carrier-frostman}
gives $\mu\in\mathcal M_1(F)$ satisfying \eqref{eq:carrier-bound}.
Summing over labelled covers gives
$\mu^*(A)\le C\mathcal P^t_\theta(A)$ for every $A\subset\R^d$.
Equation~\eqref{eq:null-sublevel} therefore implies
$\mu\{d_\theta^\mu<t\}=0$. Thus the supremum in
\eqref{eq:main} is at least $\dim_\theta^{\mathrm P}F$, also when $\dim_\theta^{\mathrm P}F=0$ by nonnegativity.
Conversely, for $\mu\in\mathcal M_1(F)$ and
$0<t<\operatorname*{ess\,inf}_{x\sim\mu}d_\theta^\mu(x)$,
apply Proposition~\ref{prop:mass-distribution} to the full-measure set
$\{x\in\supp\mu:d_\theta^\mu(x)>t\}$ to obtain $\dim_\theta^{\mathrm P}F\ge t$.
Let $t$ increase to the essential infimum and take the supremum over
$\mu$; the case of essential infimum zero is immediate.

Now let $F$ be analytic. The compact case and Proposition~\ref{prop:basic-properties}\ref{item:basic-analytic}
give
\[
\begin{aligned}
\sup_{\substack{\mu\text{ probability}\\\supp\mu\Subset F}}
       \operatorname*{ess\,inf}_{x\sim\mu}d_\theta^\mu(x)
 &=\sup_{\substack{\varnothing\ne K\subset F\\K\text{ compact}}}
          \sup_{\mu\in\mathcal M_1(K)}
          \operatorname*{ess\,inf}_{x\sim\mu}d_\theta^\mu(x)\\
 &=\sup_{\substack{\varnothing\ne K\subset F\\K\text{ compact}}}
          \dim_\theta^{\mathrm P}K
       =\dim_\theta^{\mathrm P}F.
\end{aligned}
\]
\end{proof}

For Theorem~\ref{thm:product-testing-duality}, we require
\eqref{eq:suitable-radius} with a scale threshold uniform on a compact
set of positive measure. The next lemma gives this after decreasing
the exponent.
\begin{lemma}
\label{lem:compact-uniformization}
Let $\mu$ be a compactly supported Borel probability measure,
$0\le\alpha<1$, and $0<u<v$. If $d_\alpha^\mu>v$ almost everywhere,
then there are a compact set $K\subset\supp\mu$ with $\mu(K)>0$
and a number $\delta>0$ such that the probability measure
$\sigma=\mu\restriction K/\mu(K)$ satisfies
\begin{equation}\label{eq:compact-uniformization}
 \forall x\in K\ \forall\,0<R<\delta\quad
 \exists r\in[R^{1/(1-\alpha)},R]:
 \qquad\sigma(B(x,r))\le r^u.
\end{equation}
\end{lemma}
\begin{proof}
Write $M_{R,\alpha}^\mu$ for the maximum in
\eqref{eq:window-maximum} with parameter $\alpha$.
By Lemma~\ref{lem:window-measurability}, for $N\ge2$, the set
\[
 K_N=\{x\in\supp\mu:M_{R,\alpha}^\mu(x)\ge v
                    \text{ for every }0<R<1/N\}
\]
is compact. The sets $K_N$ increase with $N$, and their union has full
$\mu$-measure because $d_\alpha^\mu>v$ almost everywhere.
Choose $N$ with $a=\mu(K_N)>0$, put $K=K_N$ and
$\sigma=a^{-1}\mu\restriction K$, and set
\[
 \delta=\min\{1/N,a^{1/(v-u)}\}>0.
\]
For $x\in K$ and $0<R<\delta$, choose a radius
$r\in[R^{1/(1-\alpha)},R]$ at which the maximum
$M_{R,\alpha}^\mu(x)$ is attained. Then $\mu(B(x,r))\le r^v$ and,
since $u<v$ and $r\le R<\delta$,
\[
 \sigma(B(x,r))\le a^{-1}r^v
 =r^u a^{-1}r^{v-u}
 \le r^u a^{-1}\delta^{v-u}\le r^u.
\]
\end{proof}

\subsection{Measure spectra and attainment at the critical exponent}

For a compactly supported Borel probability measure $\mu$, write
\[
 \begin{aligned}
 P^-_\theta(\mu)
 &=\inf\{\dim_\theta^{\mathrm P}A:A\text{ Borel},\ \mu(A)>0\},\\
 P^+_\theta(\mu)
 &=\inf\{\dim_\theta^{\mathrm P}A:A\text{ Borel},\ \mu(A)=1\}.
 \end{aligned}
\]

\begin{corollary}
\label{cor:measure-spectra}
For every compactly supported Borel probability measure $\mu$ and
$0\le\theta<1$,
\[
 P^-_\theta(\mu)
 =\operatorname*{ess\,inf}_{x\sim\mu}d_\theta^\mu(x),
 \qquad
 P^+_\theta(\mu)
 =\operatorname*{ess\,sup}_{x\sim\mu}d_\theta^\mu(x).
\]
\end{corollary}

\begin{proof}
Write $\alpha=\operatorname*{ess\,inf}_\mu d_\theta^\mu$ and
$\beta=\operatorname*{ess\,sup}_\mu d_\theta^\mu$.
For $0<t<\alpha$, every Borel set of positive $\mu$-measure has a
positive-measure intersection with $\{d_\theta^\mu>t\}$; for
$0<t<\beta$, the same holds for every Borel set of full $\mu$-measure.
Proposition~\ref{prop:mass-distribution} therefore gives
$P^-_\theta(\mu)\ge\alpha$ and $P^+_\theta(\mu)\ge\beta$. In particular,
$0\le\alpha\le\beta\le P^+_\theta(\mu)\le d$; the zero cases
follow from nonnegativity. For every $\varepsilon>0$, the Borel
sublevel sets $\{d_\theta^\mu\le\alpha+\varepsilon\}$ and
$\{d_\theta^\mu\le\beta+\varepsilon\}$ have positive and full
measure, respectively. Proposition~\ref{thm:sublevel} therefore gives
$P^-_\theta(\mu)\le\alpha+\varepsilon$ and
$P^+_\theta(\mu)\le\beta+\varepsilon$. Let $\varepsilon\downarrow0$.
\end{proof}

To prove attainment under positivity at the critical exponent, we
first establish a logarithmic refinement of
Proposition~\ref{prop:carrier-frostman}.

\begin{proposition}
\label{prop:log-carrier-frostman}
Let $F\subset\R^d$ be nonempty and compact, let $0\le\theta<1$,
and suppose that $s>0$ and $\mathcal P^s_\theta(F)>0$.
There exist $\mu\in\mathcal M_1(F)$ and $C,\delta_0>0$ such that
\begin{equation}\label{eq:log-carrier}
 \mu^*(E)\le C\bigl(1+\log(1/R)\bigr)^2\Phi^s_\theta(E,R)
 \qquad(0<R<\delta_0,\ \diam E\le R).
\end{equation}
\end{proposition}

\begin{proof}
Put $h(R)=(1+\log(1/R))^2$ and define $q_{\delta,h}$ by
\eqref{eq:weighted}, with exponent $s$ and with each cost
$\Phi^s_\theta(E_i,R_i)$ multiplied by $h(R_i)$.
Write $A_{\delta,h}=q_{\delta,h}(1)$. In the proof of
Lemma~\ref{lem:weighted-content-comparison}, use equal exponents
$u=t=s$ in \eqref{eq:grid-local} and observe that
\[
 \lambda_k^{-1}=2(k+1)^2\le C h(R_i)
 \qquad(2^{-k-1}<R_i\le2^{-k}).
\]
Summing the costs of these carriers over $k\ge K$ gives
\begin{equation}\label{eq:log-rounding}
 \mathcal P^s_{\theta,2^{1-K}}(F)\le C_{d,s}A_{2^{-K},h}.
\end{equation}
Choose $K\ge3$ with $\mathcal P^s_{\theta,2^{1-K}}(F)>0$ and set
$\delta=2^{-K}$. Then $0<A_{\delta,h}<\infty$.

Extend measures on $F$ by zero outside $F$. The Hahn--Banach and Riesz argument in
Proposition~\ref{prop:carrier-frostman}, applied to $q_{\delta,h}$,
gives a positive regular Borel measure $\nu$ on $F$ with
$\nu(F)=A_{\delta,h}$ and
$\nu(U\cap F)\le h(R)\Phi^s_\theta(U,R)$ for open labelled carriers
with $R\le\delta$. For an arbitrary carrier $(E,R)$ with $R<\delta$,
Lemma~\ref{lem:open-labelled-carriers} gives open enlargements
$(U_m,R_m)$ with $R_m\downarrow R$, $R_m<\delta$, and
$\limsup_m\Phi^s_\theta(U_m,R_m)\le\Phi^s_\theta(E,R)$.
Continuity of $h$ at $R$ yields
\[
 \nu^*(E)\le\liminf_m h(R_m)\Phi^s_\theta(U_m,R_m)
 \le h(R)\Phi^s_\theta(E,R).
\]
Normalizing by $A_{\delta,h}$ proves \eqref{eq:log-carrier}.
\end{proof}

\begin{corollary}\label{cor:positive-content}
Let $F\subset\R^d$ be nonempty and compact and let $0\le\theta<1$.
If
\[
 D=\dim_\theta^{\mathrm P}F,
 \qquad \mathcal P^D_\theta(F)>0,
\]
then there exists $\mu\in\mathcal M_1(F)$ such that
\[
 d_\theta^\mu(x)=D\qquad\text{for $\mu$-almost every }x.
\]
\end{corollary}

\begin{proof}
If $D=0$, take a Dirac measure on $F$.
If $D>0$, Proposition~\ref{prop:log-carrier-frostman} gives
$\mu\in\mathcal M_1(F)$ satisfying \eqref{eq:log-carrier} with
$s=D$. For every $0<t<D$, Lemma~\ref{lem:elementary-local-estimates}
gives
\[
 \mu^*(E)
 \le C\bigl(1+\log(1/R)\bigr)^2(2R)^{D-t}\Phi^t_\theta(E,R)
 \le C_{D,t}\Phi^t_\theta(E,R).
\]
Summing over labelled covers and using \eqref{eq:null-sublevel}
yields $d_\theta^\mu\ge t$ almost everywhere. Letting $t\uparrow D$
through a countable sequence gives $d_\theta^\mu\ge D$ almost
everywhere.
By Corollary~\ref{cor:measure-spectra},
\[
 \operatorname*{ess\,sup}_{x\sim\mu}d_\theta^\mu(x)
 =P^+_\theta(\mu)\le\dim_\theta^{\mathrm P}F=D.
\]
Thus $d_\theta^\mu=D$ almost everywhere.
\end{proof}

\section{Hausdorff products and testing sets}
\label{sec:geometric-properties}

We bound Hausdorff product dimension using the packing spectrum of one
factor and the upper intermediate dimension of the other.
Optimizing these bounds can improve the classical endpoint
estimates. Compact testing
sets yield a converse characterization of the packing spectrum. We also
establish product bounds for the spectrum itself.

\subsection{Complementary windows and Hausdorff products}

For a nonempty bounded set $F$, $t\ge0$, and $0<\eta\le1$, define
\begin{equation}\label{eq:intermediate-covering-content}
 I_\eta^t(F,R)=
 \inf\left\{\sum_i(\diam V_i)^t:
 F\subset\bigcup_iV_i,\quad R^{1/\eta}\le\diam V_i\le R\right\},
 \qquad 0<R<1.
\end{equation}
Write
\begin{equation}\label{eq:upper-intermediate-content}
 I_\eta^t(F)=\limsup_{R\downarrow0}I_\eta^t(F,R).
\end{equation}
The upper intermediate dimension is
\[
 \overline{\dim}_\eta F=\inf\{t\ge0:I_\eta^t(F)=0\},
\]
with conventions $\overline{\dim}_0F=\dimH F$ and
$\overline{\dim}_1F=\udimB F$; see~\cite{FFK2020}.
For the empty set put $I_\eta^t(\varnothing,R)=I_\eta^t(\varnothing)=0$
and $\overline{\dim}_\eta\varnothing=0$.
For $E\subset\R^d$ and $0<\eta\le1$, write
\begin{equation}\label{eq:modified-intermediate-definition}
 \overline{\dim}^{\mathrm{mod}}_\eta E
 =\inf\left\{\sup_j\overline{\dim}_\eta E_j:
       E=\bigcup_{j=1}^\infty E_j,\quad E_j\text{ bounded}\right\}.
\end{equation}
We set $\overline{\dim}^{\mathrm{mod}}_{0} E=\dimH E$;
the characterization of packing dimension by modified upper box dimension gives $\overline{\dim}^{\mathrm{mod}}_{1} E=\dimP E$.
 The modified dimensions are countably stable; see \cite[Section~2.3]{BalkaKeleti2025}, following
Douzi and Selmi~\cite{DouziSelmi2022}.

For the modified intermediate dimensions, one takes critical exponents
before minimizing over countable covers. For the packing spectrum, one
minimizes the finite-scale packing costs before taking critical exponents.

\begin{theorem}
\label{thm:direct-ffk-hausdorff-product}
\label{cor:modified-hausdorff-product}
For bounded $E\subset\R^m$, $F\subset\R^n$, and $0\le\theta\le1$,
\begin{equation}
 \label{eq:direct-ffk-hausdorff-product-symmetric}
 \dimH(E\times F)
 \le\min\{\dim_\theta^{\mathrm P}E+\overline{\dim}_{1-\theta}F,
\dim_\theta^{\mathrm P}F+\overline{\dim}_{1-\theta}E\}.
\end{equation}
For arbitrary $E\subset\R^m$ and $F\subset\R^n$, the same inequality
holds for every $0\le\theta\le1$ with each upper intermediate
dimension replaced by its modified version.
\end{theorem}

Write $\mathcal H^q_\rho$ for Hausdorff content using sets of diameter
at most $\rho$.

\begin{proposition}
\label{thm:complementary-window-content}
Let $E\subset\R^m$, $F\subset\R^n$ be bounded, let $s,t>0$, and
let $0\leq\theta<1$. Put $\eta=1-\theta$. There is a constant
$C_{s,t}>0$ such that, for $0<\delta<1$,
\begin{equation}
 \label{eq:complementary-window-content}
 \mathcal H^{s+t}_{\sqrt{17}\delta}(E\times F)
 \leq C_{s,t}
 \inf_{(U_j,R_j)}\sum_j
 \Phi^s_\theta(U_j,R_j)I_\eta^t(F,R_j),
\end{equation}
where the infimum is over labelled covers of $E$ with
$\diam U_j\leq R_j\leq\delta$.
\end{proposition}
\begin{proof}
Fix $\varepsilon>0$ and a labelled cover $(U_j,R_j)_j$ with
$U_j\ne\varnothing$ for every $j$.
Put $a_j=\Phi^s_\theta(U_j,R_j)$. By
Lemma~\ref{lem:elementary-local-estimates}, $0<a_j<\infty$.
Choose $\gamma_j>0$ with
$\sum_ja_j\gamma_j<\varepsilon$ and covers $(V_{j,k})_k$ of $F$ such that,
writing
$r_{j,k}=\diam V_{j,k}$,
\[
 R_j^{1/\eta}\leq r_{j,k}\leq R_j,
 \qquad \sum_kr_{j,k}^t\leq
I_\eta^t(F,R_j)
+\gamma_j.
\]
Choose a maximal $2r_{j,k}$-separated set
$S_{j,k}\subset\overline{U_j}$ and put $N_{j,k}=\#S_{j,k}$. Then
$N_{j,k}(2r_{j,k})^s\leq a_j$ and
\[
 U_j\times V_{j,k}\subset
 \bigcup_{x\in S_{j,k}}\bigl(B(x,2r_{j,k})\times V_{j,k}\bigr).
\]
For every $x\in S_{j,k}$,
$\diam\bigl(B(x,2r_{j,k})\times V_{j,k}\bigr)
\leq\sqrt{17}\,r_{j,k}\leq\sqrt{17}\delta$. Therefore
\begin{align*}
 \mathcal H^{s+t}_{\sqrt{17}\delta}(E\times F)
 &\leq\sum_{j,k}N_{j,k}(\sqrt{17}\,r_{j,k})^{s+t}\\
 &\leq 2^{-s}17^{(s+t)/2}\sum_ja_j\sum_kr_{j,k}^t\\
 &\leq 2^{-s}17^{(s+t)/2}\sum_ja_j
       \bigl(I_\eta^t(F,R_j)+\gamma_j\bigr).
\end{align*}
Letting $\varepsilon\downarrow0$ and taking the infimum over labelled
covers proves the estimate with $C_{s,t}=2^{-s}17^{(s+t)/2}$.
\end{proof}

\begin{proof}[Proof of Theorem~\ref{thm:direct-ffk-hausdorff-product}]
At $\theta=1$, both assertions follow from the classical
Hausdorff--packing product inequality~\cite{Howroyd1996} and the
endpoint identities.
By symmetry, it suffices to prove the first bound in each minimum.
For $0\le\theta<1$, choose
$s>\dim_\theta^{\mathrm P}E$ and $t>\overline{\dim}_{1-\theta}F$.
The definition of the upper intermediate dimension gives $\delta_0>0$
such that
\begin{equation}
 \label{eq:uniform-complementary-content-bound}
I_{1-\theta}^t(F,R)
\le1\qquad(0<R<\delta_0).
\end{equation}
For $0<\delta<\delta_0$, Proposition~\ref{thm:complementary-window-content}
therefore yields
\[
 \mathcal H^{s+t}_{\sqrt{17}\delta}(E\times F)
 \le C_{s,t}\mathcal P^s_{\theta,\delta}(E)=0.
\]
First let $\delta\downarrow0$, and then let
$s\downarrow\dim_\theta^{\mathrm P}E$ and
$t\downarrow\overline{\dim}_{1-\theta}F$.

For the modified assertion, suppose first that $E$ is bounded.
For every bounded decomposition $F=\bigcup_jF_j$, countable
stability of Hausdorff dimension and
\eqref{eq:direct-ffk-hausdorff-product-symmetric} give
\[
\dimH(E\times F)=\sup_j\dimH(E\times F_j)
 \le\dim_\theta^{\mathrm P}E+
       \sup_j\overline{\dim}_{1-\theta}F_j.
\]
Take the infimum over decompositions. For unbounded $E$, apply the
result to $E\cap[-k,k]^m$ and use countable stability.
\end{proof}

\begin{corollary}
\label{cor:hausdorff-additivity}
Let $E\subset\R^m$ and $F\subset\R^n$ be nonempty and bounded.
If
\[
 \dim_\theta^{\mathrm P}E=\dimH E\quad(0\le\theta<1),
 \qquad \lim_{\eta\downarrow0}\overline{\dim}_{\eta}F=\dimH F,
\]
then
\begin{equation}
\label{eq:exact-criterion-final}
 \dimH(E\times F)=\dimH E+\dimH F.
\end{equation}
\end{corollary}

\begin{proof}
Let $\theta=1-\eta\uparrow1$ in
\eqref{eq:direct-ffk-hausdorff-product-symmetric} to obtain the upper
bound. The reverse inequality is the classical Hausdorff product
lower bound~\cite{Howroyd1996}.
\end{proof}

Example~\ref{ex:strict-product-improvement} applies this criterion to
obtain a strict improvement over both classical endpoint bounds.

\subsection{Compact testing sets and the product characterization}
\label{sec:product-duality}

\begin{theorem}
\label{thm:product-testing-duality}
Let $n\ge1$. For every nonempty
analytic $E\subset\R^n$
and $0\le\theta\le1$,
\begin{equation}\label{eq:product-testing-duality}
 \dim_\theta^{\mathrm P}E
 =\sup_{\substack{\varnothing\ne F\subset[0,1]^n\\F\text{ compact}}}
   \bigl\{\dimH(E\times F)-\overline{\dim}_{1-\theta}F\bigr\}.
\end{equation}
The identity remains valid with
$\overline{\dim}^{\mathrm{mod}}_{1-\theta}F$ in place of
$\overline{\dim}_{1-\theta}F$.
\end{theorem}

\begin{lemma}
\label{lem:dyadic-tester-construction}
Let $K\subset(0,1)^n$ be nonempty and compact, and let $\sigma$ be
a Borel probability measure carried by $K$. Put
$k_j=j^2$, $b_j=k_j+2j$, and $c_n=2\sqrt n$.
There exist a nonempty compact set $\Gamma\subset K\times[0,1]^n$ with
$\dimH\Gamma\ge n$ and a constant $C_n>0$
such that, for every $j\ge1$ and $a\ge0$, the set
\begin{equation}\label{eq:duality-mass-projection}
 F_j(a)=\pi_2\{(x,y)\in\Gamma:
                  \sigma(B(x,c_n2^{-k_j}))\le2^{-a}\}
\end{equation}
can be covered by at most
$2^{nk_j-a+C_nj\log_2(j+1)}$ closed dyadic cubes of side $2^{-b_j}$.
Here $\pi_2$ denotes projection onto the second factor.
\end{lemma}

\begin{proof}
{Construction.}
Put $k_0=0$, $L_j=k_j-k_{j-1}=2j-1$, and $M_j=nL_j+1$ for $j\ge1$.
Use the half-open dyadic cubes in $[0,1)^n$ at depths $k_j$,
starting from $Q_0=[0,1)^n$. At each retained cube $Q$ of positive
$\sigma$-mass at depth $k_{j-1}$, its positive-mass descendants
$Q'$ at depth $k_j$ have conditional probabilities
\[
 p(Q'\mid Q)=\frac{\sigma(Q')}{\sigma(Q)},
 \qquad \sum_{Q'}p(Q'\mid Q)=1.
\]
Retain those with $p(Q'\mid Q)\ge2^{-M_j}$. There are at most
$2^{nL_j}$ descendants, so the retained conditional probabilities
have sum at least $1/2$. On each retained edge set
\[
 a(Q'\mid Q)=\lfloor-\log_2p(Q'\mid Q)\rfloor
       \in\{0,\ldots,M_j\}.
\]

For each $j$, choose sets $W_{j,a}$ of words
of length $L_j+2$ over $\{0,1\}^n$ such that
\[
 \#W_{j,a}=2^{M_j-a}\quad(0\le a\le M_j).
\]
For each retained cube $Q$ at depth $k_{j-1}$, let $N_j(Q)$ be the
number of pairs $(Q',w)$ with $Q'$ a retained descendant of $Q$ at
depth $k_j$ and $w\in W_{j,a(Q'\mid Q)}$. Then
\[
 N_j(Q)=\sum_{Q'\text{ retained}}2^{M_j-a(Q'\mid Q)}
 \ge2^{M_j}\sum_{Q'\text{ retained}}p(Q'\mid Q)
 \ge2^{nL_j}.
\]
Let $\Omega$ consist of all sequences $\omega=((Q_j,w_j))_{j\ge1}$
such that $(Q_j)_{j\ge0}$ is a retained path and
$w_j\in W_{j,a(Q_j\mid Q_{j-1})}$ for every $j\ge1$.
This space is nonempty and compact, as a closed subset of a product
of finite alphabets. The concatenation $w_1\cdots w_j$ determines a
closed dyadic cube $P_j(\omega)$ of side $2^{-b_j}$, since
$\sum_{i=1}^j(L_i+2)=b_j$.
Define $\Psi:\Omega\to\R^{2n}$ by
\[
 \{\Psi(\omega)\}
 =\bigcap_{j=1}^{\infty}\bigl(\overline{Q_j}\times P_j(\omega)\bigr).
\]
The rectangles are nested and their diameters tend uniformly to zero,
so $\Psi$ is continuous. Each $Q_j$ meets $K$, so
$\Gamma=\Psi(\Omega)$ is a nonempty compact subset of $K\times[0,1]^n$.

{Hausdorff dimension estimate.}
For $\omega\in\Omega$, let $C_j(\omega)$ be the set of sequences
in $\Omega$ agreeing with $\omega$ in the first $j$ pairs $(Q_i,w_i)$.
For each path prefix ending at $Q$ at depth $k_{j-1}$, assign
conditional probability $1/N_j(Q)$ to each admissible next pair
$(Q',w)$.
This defines a probability $\widetilde\lambda$ on $\Omega$, with
\begin{equation}\label{eq:duality-joint-cylinder}
 \widetilde\lambda(C_j(\omega))
 =\prod_{i=1}^j\frac{1}{N_i(Q_{i-1})}
 \le\prod_{i=1}^j2^{-nL_i}=2^{-nk_j}.
\end{equation}
Put $\lambda=\Psi_*\widetilde\lambda$, a Borel probability measure
supported on $\Gamma$.
Since $b_j\ge k_j$, if $2^{-b_{j+1}}<r\le2^{-b_j}$,
the ball $B(z,r)\subset\R^{2n}$ meets at most $C(n)$ of the
level-$j$ rectangles $\overline{Q_j}\times P_j$.
By \eqref{eq:duality-joint-cylinder}, for each $0<t<n$,
\[
 \lambda(B(z,r))
 \le C(n)2^{-nk_j}\le C_t r^t
 \qquad(z\in\R^{2n}),
\]
because $nk_j-tb_{j+1}\to\infty$. The mass distribution principle
gives $\dimH\Gamma\ge n$.

{Projection covering estimate.}
For $\omega\in\Omega$, write $a_i=a(Q_i\mid Q_{i-1})$ and
$A_j(\omega)=\sum_{i=1}^j a_i$.
Fix $j\ge1$ and $a\ge0$, and let $(x,y)\in\Gamma$ satisfy
$\sigma(B(x,c_n2^{-k_j}))\le2^{-a}$. Choose $\omega\in\Omega$
with $\Psi(\omega)=(x,y)$. Since $x\in\overline{Q_j}$ and
$c_n>\sqrt n$, we have $Q_j\subset B(x,c_n2^{-k_j})$. Hence
\begin{equation}\label{eq:duality-label-mass}
 \begin{aligned}
 A_j(\omega)
 &\ge\sum_{i=1}^j-\log_2p(Q_i\mid Q_{i-1})-j\\
 &=-\log_2\sigma(Q_j)-j\\
 &\ge a-j.
 \end{aligned}
\end{equation}
Thus the family
\[
 \mathcal U_j(a)=\{P_j(\omega):\omega\in\Omega,\ A_j(\omega)\ge a-j\}
\]
covers $F_j(a)$. For a fixed label sequence
$(a_1,\ldots,a_j)$, the number of distinct words $w_1\cdots w_j$
is at most
\[
 \prod_{i=1}^j\#W_{i,a_i}=2^{nk_j+j-\sum_{i=1}^j a_i}.
\]
Consequently,
\begin{equation}\label{eq:duality-prefix-count}
 \#\mathcal U_j(a)
 \le\left(\prod_{i=1}^j(M_i+1)\right)2^{nk_j+2j-a}
 \le2^{nk_j-a+C_nj\log_2(j+1)},
\end{equation}
where $C_n$ depends only on $n$.
\end{proof}

\begin{lemma}
\label{lem:uniform-window-tester}
Let $K\subset(0,1)^n$ be nonempty and compact, and let $\sigma$ be
a Borel probability measure carried by $K$. Let $\Gamma$ be as in
Lemma~\ref{lem:dyadic-tester-construction} for
$(K,\sigma)$, and put $F=\pi_2\Gamma$.
Suppose that $0<s<u<n$ and
$1\le p_0<p<\infty$, and that there exists $R_0>0$ such that
\begin{equation}\label{eq:duality-uniform-window}
 \forall x\in K\quad\forall\,0<R<R_0\quad
 \exists r\in[R^{p_0},R]:\qquad \sigma(B(x,r))\le r^u.
\end{equation}
Then
\[
 \overline{\dim}_{1/p}F\le n-s,
 \qquad \dimH(K\times F)\ge n.
\]
In particular, the choice of $F$ is independent of $s,u,p_0,p$.
\end{lemma}

\begin{proof}
Retain the notation of Lemma~\ref{lem:dyadic-tester-construction}.
Since $\Gamma\subset K\times F$, we have
$\dimH(K\times F)\ge\dimH\Gamma\ge n$.
Choose $s<v<u$ and $\delta>0$ with $p_0(1+\delta)<p$.

Let $R=2^{-T}$, where $T\in\R$ is sufficiently large.
For every $(x,y)\in\Gamma$, apply
\eqref{eq:duality-uniform-window} at outer scale $2^{-(1+\delta)T}$.
There is $r=2^{-q}$ such that
\[
 (1+\delta)T\le q\le p_0(1+\delta)T,
 \qquad \sigma(B(x,2^{-q}))\le2^{-uq}.
\]
Let $j$ be the least integer with $k_j\ge q+\log_2c_n$.
Then $c_n2^{-k_j}\le2^{-q}$ and
$0\le k_j-q\le\log_2c_n+2j-1$. Since $j=O(\sqrt T)$,
\[
 \sigma(B(x,c_n2^{-k_j}))
 \le2^{-uq}\le2^{-vk_j}
\]
for all sufficiently large $T$, uniformly in $(x,y)$.
Thus $y\in F_j(vk_j)$. The covering cubes supplied by
Lemma~\ref{lem:dyadic-tester-construction} have diameter
$\Delta_j=\sqrt n\,2^{-b_j}$, and
\[
 (1+\delta)T-O_n(1)
 \le b_j-\log_2\sqrt n
 \le p_0(1+\delta)T+O(\sqrt T).
\]
It follows that $R^p\le\Delta_j\le R$ for all sufficiently large
$T$.

The finite index set
$\mathcal J_R=\{j\ge1:R^p\le\Delta_j\le R\}$ therefore satisfies
\[
 F\subset\bigcup_{j\in\mathcal J_R}F_j(vk_j).
\]
Cover each $F_j(vk_j)$ by the cubes from
Lemma~\ref{lem:dyadic-tester-construction}. Their total
$(n-s)$-cost is at most
\begin{align*}
 \sum_{j\in\mathcal J_R}
       2^{(n-v)k_j+C_nj\log_2(j+1)}\Delta_j^{n-s}
 &\le C\sum_{j:\ k_j\ge T/2}
       2^{-(v-s)k_j+C_nj\log_2(j+1)}\\
 &\le C\sum_{j:\ k_j\ge T/2}2^{-(v-s)k_j/2}
 \longrightarrow0.
\end{align*}
Here $\Delta_j\le2^{-T}$ implies $k_j\ge T/2$ for large $T$,
since $b_j/k_j\to1$. These covers are admissible at every
sufficiently small outer scale, so $\overline{\dim}_{1/p}F\le n-s$.
\end{proof}

\begin{proof}[Proof of Theorem~\ref{thm:product-testing-duality}]
Theorem~\ref{thm:direct-ffk-hausdorff-product} bounds the
supremum in \eqref{eq:product-testing-duality} from above by
$\dim_\theta^{\mathrm P}E$.
For unbounded $E$, the same upper bound follows by
bounded exhaustion and countable stability of Hausdorff dimension.
If $\theta=0$ or $\dim_\theta^{\mathrm P}E=0$, the singleton
$F=\{0\}$ gives equality, since
$\dimH E=\dim_\theta^{\mathrm P}E$ in either case. For the lower bound when $0<\theta\le1$ and
$\dim_\theta^{\mathrm P}E>0$, suppose first that $E$ is compact.

Suppose $0<\theta<1$ and $0<s<\dim_\theta^{\mathrm P}E$.
By bi-Lipschitz invariance, we may translate and scale $E$ so that
$E\subset(0,1)^n$.
By continuity of the packing spectrum, choose $0\le\alpha<\theta$
and exponents
\[
 s<u<v<\dim_\alpha^{\mathrm P}E\le n,
 \qquad p_0=(1-\alpha)^{-1}<p=(1-\theta)^{-1}.
\]
By Theorem~\ref{thm:main}, there exists $\mu\in\mathcal M_1(E)$ with
$d_\alpha^\mu>v$ almost everywhere.
By Lemma~\ref{lem:compact-uniformization}, applied at parameter
$\alpha$ with exponents $u<v$, there exist a compact set
$K\subset\supp\mu\subset E$ with $\mu(K)>0$ and a number $R_0>0$
such that the probability measure
$\sigma=\mu\restriction K/\mu(K)$ satisfies
\[
 \forall x\in K\quad\forall\,0<R<R_0\quad
 \exists r\in[R^{p_0},R]:\qquad \sigma(B(x,r))\le r^u.
\]
Choose $\Gamma$ for $(K,\sigma)$ as in
Lemma~\ref{lem:dyadic-tester-construction} and put $F=\pi_2\Gamma$.
Lemma~\ref{lem:uniform-window-tester} gives
$\overline{\dim}_{1-\theta}F\le n-s$ and
$\dimH(E\times F)\ge\dimH(K\times F)\ge n$.

For $\theta=1$ and $0<s<\dimP E$, Keleti--M\'ath\'e~\cite[Theorem~1.9]{KeletiMathe2023}
give a compact $B\subset\R^n$ with $\mathcal H^{n-s}(B)=0$ and
$\operatorname{int}(E+B)\ne\varnothing$. Lipschitz addition yields
$\dimH(E\times B)\ge\dimH(E+B)=n$, while $\dimH B\le n-s$.
A similarity sends $B$ into $[0,1]^n$ and preserves both dimensions.

For $0<\theta\le1$ and $\dim_\theta^{\mathrm P}E>0$,
the preceding constructions give
\begin{align*}
 \sup_{\substack{\varnothing\ne F\subset[0,1]^n\\F\text{ compact}}}
   \bigl\{\dimH(E\times F)-\overline{\dim}_{1-\theta}F\bigr\}
 \ge\sup_{0<s<\dim_\theta^{\mathrm P}E}s
 =\dim_\theta^{\mathrm P}E.
\end{align*}
Thus \eqref{eq:product-testing-duality} holds for compact $E$.

For a general nonempty analytic $E$ and
$0<s<\dim_\theta^{\mathrm P}E$, compact determination
(Proposition~\ref{prop:basic-properties}\ref{item:basic-analytic})
gives a nonempty compact $K\subset E$ with
$\dim_\theta^{\mathrm P}K>s$.
Applying \eqref{eq:product-testing-duality} to $K$ and using
$\dimH(E\times F)\ge\dimH(K\times F)$ for every nonempty compact set
$F\subset[0,1]^n$ shows that the supremum for $E$ is at least
$\dim_\theta^{\mathrm P}K>s$. Letting
$s\uparrow\dim_\theta^{\mathrm P}E$ proves the lower bound.

Finally, Theorem~\ref{thm:direct-ffk-hausdorff-product}
bounds the supremum involving the modified intermediate dimension
from above by $\dim_\theta^{\mathrm P}E$. Since
$\overline{\dim}^{\mathrm{mod}}_{1-\theta} F\le\overline{\dim}_{1-\theta}F$ for bounded $F$,
\eqref{eq:product-testing-duality} bounds it below by the same
quantity.
\end{proof}

Zindulka~\cite{Zindulka2013}
restricts packing radii to a fixed global set of scales and characterizes
directed lower packing dimension by an infimum of
$\dimP(E\times F)-\dimP F$, attained with a compact factor $F$.
Here the packing window depends on the carrier label, and the product
characterization is a supremum involving Hausdorff dimension and the upper
intermediate dimension of the testing set.

\subsection{Product bounds for the packing spectrum}

\begin{theorem}
\label{cor:mixed-product-packing-open}
\label{thm:spectrum-product-lower}
For $E\subset\R^m$, $F\subset\R^n$, and $0\le\theta\le1$,
\begin{equation}
\label{eq:product-packing-upper-symmetric}
 \dim_\theta^{\mathrm P}(E\times F)
 \le\min\{\dim_\theta^{\mathrm P}E+\dimP F,
             \dim_\theta^{\mathrm P}F+\dimP E\}.
\end{equation}
If, in addition, $E$ and $F$ are nonempty and analytic, then
\begin{equation}
\label{eq:spectrum-product-lower}
 \max\{\dim_\theta^{\mathrm P}E+\dimH F,
          \dimH E+\dim_\theta^{\mathrm P}F\}
 \le\dim_\theta^{\mathrm P}(E\times F).
\end{equation}
\end{theorem}

For bounded $A\subset\R^d$ and $r>0$, write
\[
 M_r(A)=\sup\{\#S:S\subset\overline A\text{ is }2r\text{-separated}\},
 \qquad N_r^*(A)=N_r(\overline A),
\]
where $N_r(A)$ is the least number of open balls of radius $r$ covering
$A$.

\begin{lemma}
\label{lem:product-carrier-estimate}
Let $m,n\ge1$ be integers, let $F\subset\R^n$ be nonempty and bounded,
and suppose that there are $T\ge0$ and $C_F,r_0>0$ such that
\begin{equation}
\label{eq:box-cover-open}
 N_r(F)\le C_Fr^{-T}\qquad(0<r<r_0).
\end{equation}
Fix $0<\theta<1$ and $0<\sigma<s$, and put
$p=(1-\theta)^{-1}$, $q=s+T$, and $C_0=\sqrt{n+1}$.
Let $\mathcal Q(R)$ be the grid of half-open cubes of side $R$ in
$\R^n$, and write $F_Q=F\cap Q$.
There exist $C>0$ and $0<R_*<\min\{1,r_0\}/C_0$, depending only on
$m,n,\theta,\sigma,s,T,C_F,r_0$, such that every nonempty labelled
carrier $(U,R)$ in $\R^m$ with $\diam U\le R$ and $0<R<R_*$ satisfies
\begin{equation}
\label{eq:tile-sum-direct}
 \sum_{Q\in\mathcal Q(R)}\Phi^q_\theta(U\times F_Q,C_0R)
 \le C R^{s-\sigma}\Phi^\sigma_\theta(U,R).
\end{equation}
\end{lemma}
\begin{proof}
We first record two elementary counting estimates.
For bounded $U\subset\R^m$, $V\subset\R^n$, and $r>0$,
\begin{equation}
 \label{eq:prod-pack-open}
 M_r(U\times V)\leq C_mM_r(U)N_{r/2}^*(V),
\end{equation}
where $C_m$ depends only on $m$.

Grouping a $2r$-separated subset of $\overline U\times\overline V$
according to an $r/2$-ball cover of $\overline V$ gives
$\sqrt3\,r$-separated first coordinates in each group. A volume
comparison with a cover of $\overline U$ centred on a maximal
$2r$-separated set bounds the size of each group by $C_mM_r(U)$.

Let $\mathcal Q(R)$ be the grid of half-open cubes of side length $R$
in $\R^n$, and let $F\subset\R^n$ be bounded. For every $c>0$,
\begin{equation}
 \label{eq:sum-local-cover-open}
 \sum_{Q\in\mathcal Q(R)}N_r^*(F\cap Q)
 \leq C(n,c)N_{r/2}(F)
 \qquad(0<r\leq cR).
\end{equation}

Indeed, enlarging a finite $r/2$-ball cover of $F$ to radius $r$
covers each $\overline{F\cap Q}$ using the balls that meet $F\cap Q$,
and each original ball meets at most $C(n,c)$ grid cubes.

We now prove the carrier estimate.
Choose $R_*>0$ such that $C_0R_*<\min\{1,r_0\}$, and fix a carrier
$(U,R)$ as in the statement. Put $S=C_0R$. For $Q\in\mathcal Q(R)$, the label $S$ is
admissible because
\begin{equation}
 \label{eq:tile-diam-bound-open}
 \diam(U\times F_Q)\leq\sqrt{R^2+nR^2}=S.
\end{equation}
Let $L=\lceil\log_2(S/S^p)\rceil$ and set
\[
 b_\ell=2^{-\ell}S,\qquad
 a_\ell=\max\{S^p,2^{-\ell-1}S\}
 \qquad(0\leq\ell<L).
\]
Then
\[
 [S^p,S]=\bigcup_{\ell=0}^{L-1}[a_\ell,b_\ell],
 \qquad b_\ell\leq2a_\ell.
\]
In an admissible packing, the centres of balls with radii in
$[a_\ell,b_\ell]$ are $2a_\ell$-separated. Hence
\[
 \Phi^q_\theta(U\times F_Q,S)
 \leq\sum_{\ell=0}^{L-1}M_{a_\ell}(U\times F_Q)(2b_\ell)^q.
\]
Since $a_\ell\leq C_0R$, \eqref{eq:prod-pack-open} and
\eqref{eq:sum-local-cover-open}, together with \eqref{eq:box-cover-open}, give
\[
 \sum_Q \Phi^q_\theta(U\times F_Q,S)
 \leq C\sum_\ell b_\ell^q M_{a_\ell}(U)a_\ell^{-T}.
\]
Moreover, $R^p\leq S^p\leq a_\ell\leq C_0R$, and hence
\begin{equation}
 \label{eq:P-to-Phi-open}
 M_{a_\ell}(U)(2a_\ell)^\sigma
 \leq C_0^\sigma \Phi^\sigma_\theta(U,R).
\end{equation}
For $a_\ell\leq R$, equal-radius packings give
$M_{a_\ell}(U)(2a_\ell)^\sigma\leq\Phi^\sigma_\theta(U,R)$.
For $a_\ell>R$, $\diam U\leq R$ gives $M_{a_\ell}(U)=1$, so
\[
 M_{a_\ell}(U)(2a_\ell)^\sigma
 \leq C_0^\sigma(2R)^\sigma
 \leq C_0^\sigma\Phi^\sigma_\theta(U,R).
\]
Since $s>\sigma$ and the upper endpoints $b_\ell$ decrease geometrically,
\[
 \sum_{Q\in\mathcal Q(R)}\Phi^q_\theta(U\times F_Q,S)
 \leq C \Phi^\sigma_\theta(U,R)
             \sum_\ell b_\ell^{s-\sigma}
 \leq C R^{s-\sigma}\Phi^\sigma_\theta(U,R).
\]
\end{proof}

\begin{proof}[Proof of Theorem~\ref{cor:mixed-product-packing-open}]
The endpoint identities reduce the cases $\theta=0,1$ to the classical
product inequalities~\cite{Howroyd1996}. Assume $0<\theta<1$.

\emph{Upper bound.}
For bounded $E$ and $F$, we first prove
\begin{equation}
\label{eq:mixed-product-upper-symmetric}
 \dim_\theta^{\mathrm P}(E\times F)
 \le\min\{\dim_\theta^{\mathrm P}E+\udimB F,
\dim_\theta^{\mathrm P}F+\udimB E\}.
\end{equation}
By symmetry, it suffices to prove the first bound.
We may assume that both factors are nonempty.

Choose
\[
 \dim_\theta^{\mathrm P}E<\sigma<s,\qquad
 T>\udimB F,\qquad q=s+T.
\]
The definition of upper box dimension gives $C_F,r_0>0$ satisfying
\eqref{eq:box-cover-open}. Apply Lemma~\ref{lem:product-carrier-estimate}
with these parameters, let $C,R_*$ be its constants, and put
$C_0=\sqrt{n+1}$.

Since $\mathcal P^{\sigma}_{\theta}(E)=0$, for every
$0<\delta<R_*$ and every $\varepsilon>0$ there is a labelled cover
$(U_j,R_j)$ of $E$ with $R_j\leq\delta$ and
\[
 \sum_j\Phi^\sigma_\theta(U_j,R_j)<\varepsilon.
\]
The labelled carriers $(U_j\times(F\cap Q),C_0R_j)$, with
$Q\in\mathcal Q(R_j)$, cover $E\times F$. By \eqref{eq:tile-sum-direct},
\[
 \mathcal P^{q}_{\theta,C_0\delta}(E\times F)
 \leq\sum_j\sum_{Q\in\mathcal Q(R_j)}
       \Phi^q_\theta(U_j\times(F\cap Q),C_0R_j)
 \leq C\delta^{s-\sigma}\sum_j\Phi^\sigma_\theta(U_j,R_j)
 <C\delta^{s-\sigma}\varepsilon.
\]
Letting $\varepsilon\downarrow0$ and then $\delta\downarrow0$ gives
$\dim_\theta^{\mathrm P}(E\times F)\leq s+T$.
Let $s\downarrow\dim_\theta^{\mathrm P}E$ and $T\downarrow\udimB F$.

Tricot's modified upper box characterization~\cite{Tricot1982} gives,
for every $F\subset\R^n$,
\begin{equation}
 \label{eq:pack-inf-sup-box-open}
 \dimP F=\inf\left\{\sup_i\udimB F_i:
 F\subset\bigcup_{i=1}^\infty F_i,\quad F_i\text{ bounded}\right\}.
\end{equation}
Suppose first that $E$ is bounded. Given $t>\dimP F$,
\eqref{eq:pack-inf-sup-box-open} gives a countable cover
$F\subset\bigcup_iF_i$ by bounded sets with $\udimB F_i<t$.
By \eqref{eq:mixed-product-upper-symmetric} and countable stability,
\[
 \dim_\theta^{\mathrm P}(E\times F)
 \le\sup_i\dim_\theta^{\mathrm P}(E\times F_i)
 \le\dim_\theta^{\mathrm P}E+t.
\]
Let $t\downarrow\dimP F$. For arbitrary $E$, apply this conclusion
to its bounded pieces and use countable stability once more:
\begin{align*}
 \dim_\theta^{\mathrm P}(E\times F)
 &=\sup_{k\ge1}\dim_\theta^{\mathrm P}
       \bigl((E\cap[-k,k]^m)\times F\bigr)\\
 &\le\sup_{k\ge1}\dim_\theta^{\mathrm P}(E\cap[-k,k]^m)+\dimP F\\
 &=\dim_\theta^{\mathrm P}E+\dimP F.
\end{align*}

\emph{Lower bound.}
By symmetry, it suffices to prove
$\dim_\theta^{\mathrm P}E+\dimH F\le\dim_\theta^{\mathrm P}(E\times F)$. If either summand is zero,
the assertion follows from an isometric copy of the other factor in
$E\times F$ and $\dimH\le\dim_\theta^{\mathrm P}$.
Otherwise choose
\[
 0<a<\dim_\theta^{\mathrm P}E,
 \qquad 0<b<\dimH F.
\]
Theorem~\ref{thm:main} gives a probability measure $\mu$ with compact
support in $E$ and $d_\theta^\mu>a$ almost everywhere. Compact
determination of Hausdorff dimension and the classical Frostman lemma
give a probability measure $\nu$ with compact support in $F$ and a
constant $C\ge1$ such that
\[
 \nu(B(y,r))\le Cr^b
 \qquad(y\in\R^n,\ 0<r<1).
\]
Put $\lambda=\mu\times\nu$. For $x\in\supp\mu$, $y\in\supp\nu$,
and $0<r<1$,
\[
 \lambda(B((x,y),r))\le\mu(B(x,r))\nu(B(y,r))
 \le C\mu(B(x,r))r^b,
\]
and hence
\[
 \frac{\log\lambda(B((x,y),r))}{\log r}
 \ge\frac{\log\mu(B(x,r))}{\log r}+b
       -\frac{\log C}{|\log r|}.
\]
The error is bounded by $\log C/|\log R|$ uniformly over
$R^{1/(1-\theta)}\le r\le R$. Taking the window supremum and then
the lower limit gives $d_\theta^\lambda(x,y)\ge d_\theta^\mu(x)+b$,
so $d_\theta^\lambda>a+b$ almost everywhere.
Since $\supp\lambda\Subset E\times F$, Theorem~\ref{thm:main}
implies $\dim_\theta^{\mathrm P}(E\times F)\ge a+b$.
Let $a\uparrow\dim_\theta^{\mathrm P}E$ and $b\uparrow\dimH F$.
\end{proof}

Consequently, if $E\subset\R^m$ and $F\subset\R^n$
are nonempty analytic sets with $\dimH F=\dimP F=t$, then
\begin{equation}\label{eq:product-exact-shift}
 \dim_\theta^{\mathrm P}(E\times F)
 =\dim_\theta^{\mathrm P}E+t
 \qquad(0\le\theta\le1).
\end{equation}
In particular, an Ahlfors regular factor shifts the entire spectrum by
its dimension.

\section{Complementary sets and attainable profiles}
\label{sec:Complementarity}

We characterize attainable packing spectra and realize every admissible
profile by a binary digit set. Complementary digit positions yield
reflected upper intermediate-dimension profiles. We begin with scale
formulas under a spatial mass comparison assumption.

\subsection{Homogeneous window formulas}

Throughout this subsection, let $\mu$ be a Borel probability
measure with $\supp\mu=K\subset\R^d$. Assume that there exist
$L\ge1$ and $r_0>0$ such that
\begin{equation}\label{eq:homogeneous}
 L^{-1}\le\frac{\mu(B(x,r))}{\mu(B(y,r))}\le L
 \qquad(x,y\in K,\ 0<r<r_0).
\end{equation}
Fix $x_0\in K$ and write
\begin{equation}\label{eq:homogeneous-mass-notation}
 m(r)=\mu(B(x_0,r)),\qquad
 b(r)=\frac{\log m(r)}{\log r}\qquad(0<r<1).
\end{equation}
For $0<R<1$, put
\[
 I_\theta(R)=
 \begin{cases}
 [R^{1/(1-\theta)},R],&0\le\theta<1,\\
 (0,R],&\theta=1,
 \end{cases}
 \qquad
 A_{s,\theta}(R)=\sup_{r\in I_\theta(R)}\frac{r^s}{m(r)}
 \quad(s\ge0).
\]
We write $\mu^*$ for the outer measure induced by $\mu$.
The hypothesis \eqref{eq:homogeneous} is the spatial part of the
homogeneity condition of L\"u, Lou, Wen and
Xi~\cite{LuLouWenXi2015}; their additional two-sided scale-growth
condition is not imposed here.

\begin{lemma}
\label{lem:homogeneous-scale-estimates}
Under \eqref{eq:homogeneous}, for $s,t\ge0$,
$0\le\theta\le1$, and $0<\eta\le1$,
\begin{align}
 \mathcal P^s_\theta(K)
 &\asymp\liminf_{R\downarrow0} \sup_{R^{1/(1-\theta)} \le r \le R} \quad
\frac{r^s}{m(r)},
 \label{eq:homogeneous-outer-measure}\\
 I_\eta^t(K)
 &\asymp\limsup_{R\downarrow0}
       \inf_{R^{1/\eta}\le r\le R}\frac{r^t}{m(r)}.
 \label{eq:homogeneous-intermediate-content}
\end{align}
The comparison constants depend only on $d,L$ and the relevant
exponent $s$ or $t$.
\end{lemma}

\begin{proof}
Fix $0<a<\min\{r_0,1\}$. Every finite $2a$-separated subset
of $K$ has at most $L/m(a)$ points, since its disjoint $a$-balls
each have mass at least $L^{-1}m(a)>0$. An unbounded set has
arbitrarily large finite $2a$-separated subsets, so $K$ is bounded.
Since $K=\supp\mu$ is closed, it is compact.

We first obtain uniform estimates at a single scale. For
$0<r<\min\{r_0,1/2\}$, choose a maximal $r$-separated subset
of $K\cap B(x_0,2r)$. Its disjoint $r/2$-balls lie in
$B(x_0,5r/2)$, so it has at most $5^d$ points. Its $r$-balls
cover $K\cap B(x_0,2r)$ and each has mass at most $Lm(r)$.
Thus, with $D=5^dL$,
\begin{equation}\label{eq:homogeneous-doubling}
 m(2r)\le Dm(r)
 \qquad(0<r<\min\{r_0,1/2\}).
\end{equation}
Two consequences will be used below. If $V$ has sufficiently
small diameter $r>0$ and meets $K$, then $V\subset B(x,2r)$
for $x\in V\cap K$, whence
\[
 \mu^*(V\cap K)\le Lm(2r)\le LDm(r).
\]
Also, a finite maximal $r/2$-separated subset
$\{x_i\}_{i=1}^N$ of $K$ gives a cover by balls $B(x_i,r/2)$
of diameter $r$. The balls $B(x_i,r/4)$ are disjoint, so
\[
 N\le\frac{L}{m(r/4)}\le\frac{LD^2}{m(r)}.
\]

For the packing comparison, fix a sufficiently small $R$ and
a carrier $(E,R)$ with $E\cap K\ne\varnothing$. For each
$r\in I_\theta(R)$, a finite maximal $2r$-separated subset
$\{x_i\}_{i=1}^N$ of $\overline{E\cap K}\subset K$ satisfies
\[
 \mu^*(E\cap K)\le\sum_{i=1}^N\mu(B(x_i,2r))
 \le LDNm(r).
\]
Its $r$-balls form an admissible packing for $\Phi^s_\theta(E,R)$.
Hence
\begin{equation}\label{eq:homogeneous-carrier-lower}
 \Phi^s_\theta(E,R)
 \ge c\mu^*(E\cap K)\frac{r^s}{m(r)}
 \qquad(r\in I_\theta(R)),
\end{equation}
which also holds when $E\cap K=\varnothing$.
For the upper bound, let $Q$ range over the half-open grid cubes
of side $R/\sqrt d$ meeting $K$, and put
$Q^{(R)}=\{z:\dist(z,Q)<R\}$. If $A_{s,\theta}(R)<\infty$,
every admissible packing for $\Phi^s_\theta(K\cap Q,R)$ has
centres in $K$ and disjoint open balls contained in $Q^{(R)}$,
so
\[
 \sum_i(2r_i)^s
 \le 2^sL A_{s,\theta}(R)\sum_i\mu(B(x_i,r_i))
 \le 2^sL A_{s,\theta}(R)\mu(Q^{(R)}).
\]
The sets $Q^{(R)}$ have overlap bounded in terms of $d$.
Taking suprema and summing therefore gives
\begin{equation}\label{eq:homogeneous-grid-upper}
 \sum_Q\Phi^s_\theta(K\cap Q,R)\le C A_{s,\theta}(R),
\end{equation}
which is immediate if $A_{s,\theta}(R)=\infty$.

For sufficiently small $\delta>0$, set
$a_\delta=\inf_{0<R\le\delta}A_{s,\theta}(R)$.
The grid covers in \eqref{eq:homogeneous-grid-upper} give
$\mathcal P^s_{\theta,\delta}(K)\le Ca_\delta$.
Conversely, for any finite $0<a<a_\delta$,
\eqref{eq:homogeneous-carrier-lower} implies that every labelled
$\delta$-cover $(E_j,R_j)$ of $K$ satisfies
\[
 \sum_j\Phi^s_\theta(E_j,R_j)
 \ge ca\sum_j\mu^*(E_j\cap K)\ge ca.
\]
Taking the infimum over covers and then the supremum over $a$
yields
\[
 ca_\delta\le\mathcal P^s_{\theta,\delta}(K)\le Ca_\delta,
\]
with the lower bound trivial when $a_\delta=0$.
Letting $\delta\downarrow0$ proves
\eqref{eq:homogeneous-outer-measure}.

For the covering comparison, the diameter estimate above gives,
for every admissible cover $(V_i)$ of $A\subset K$ with
$r_i=\diam V_i\in[R^{1/\eta},R]$,
\[
 \sum_i r_i^t
 \ge\frac1{LD}\sum_i\mu^*(V_i\cap K)\frac{r_i^t}{m(r_i)}
 \ge\frac{\mu^*(A)}{LD}
       \inf_{R^{1/\eta}\le r\le R}\frac{r^t}{m(r)}.
\]
Consequently,
\begin{equation}\label{eq:homogeneous-subset-covering-lower}
 I_\eta^t(A,R)\ge c\mu^*(A)
       \inf_{R^{1/\eta}\le r\le R}\frac{r^t}{m(r)}
 \qquad(A\subset K).
\end{equation}
For each $r\in[R^{1/\eta},R]$, the cover of $K$ by balls of
diameter $r$ constructed above is admissible and has $t$-cost
at most $LD^2r^t/m(r)$. Taking the infimum over $r$ and applying
\eqref{eq:homogeneous-subset-covering-lower} with $A=K$ yields
\begin{equation}\label{eq:homogeneous-finite-covering-content}
 I_\eta^t(K,R)
 \asymp\inf_{R^{1/\eta}\le r\le R}\frac{r^t}{m(r)}.
\end{equation}
Taking the upper limit as $R\downarrow0$ proves
\eqref{eq:homogeneous-intermediate-content}.
\end{proof}

The critical exponents of these comparisons give the two window
dimension formulas.

\begin{theorem}
\label{thm:homogeneous-window-formula}
Under \eqref{eq:homogeneous},
\begin{align}
 \dim_\theta^{\mathrm P}K
 &=\liminf_{R\downarrow0}
       \sup_{R^{1/(1-\theta)}\le r\le R}b(r),
 &0<\theta<1,\label{eq:homP}\\
 \overline{\dim}_{\eta}K
 &=\limsup_{R\downarrow0}
       \inf_{R^{1/\eta}\le r\le R}b(r),
 &0<\eta\le1.\label{eq:homI}
\end{align}
\end{theorem}

\begin{proof}
Iteration of \eqref{eq:homogeneous-doubling} bounds $b(r)$ above
for small $r$. Hence the right-hand sides
\[
 \alpha_\theta=\liminf_{R\downarrow0}
       \sup_{r\in I_\theta(R)}b(r),\qquad
 \beta_\eta=\limsup_{R\downarrow0}
       \inf_{R^{1/\eta}\le r\le R}b(r)
\]
are finite and nonnegative.
If $s>\alpha_\theta$, choose $\alpha_\theta<u<s$.
There are $R_k\downarrow0$ such that $b(r)<u$ throughout
$I_\theta(R_k)$, and hence
\[
 A_{s,\theta}(R_k)
 =\sup_{r\in I_\theta(R_k)}r^{s-b(r)}
 \le R_k^{s-u}\longrightarrow0.
\]
Equation~\eqref{eq:homogeneous-outer-measure} gives
$\mathcal P^s_\theta(K)=0$.
If $0\le s<\alpha_\theta$, choose $s<u<\alpha_\theta$.
Every sufficiently small window contains $r$ with $b(r)>u$, so
\[
 A_{s,\theta}(R)\ge r^{s-b(r)}\ge R^{s-u}
 \longrightarrow\infty.
\]
The same comparison gives $\mathcal P^s_\theta(K)=\infty$,
proving \eqref{eq:homP}.

If $t>\beta_\eta$, choose $\beta_\eta<u<t$.
Every sufficiently small window contains $r$ with $b(r)<u$, and
therefore
\[
 \inf_{R^{1/\eta}\le r\le R}\frac{r^t}{m(r)}
 \le R^{t-u}\longrightarrow0.
\]
Equation~\eqref{eq:homogeneous-intermediate-content} gives
$I_\eta^t(K)=0$.
If $0\le t<\beta_\eta$, choose $t<u<\beta_\eta$.
There are $R_k\downarrow0$ such that $b(r)>u$ throughout
$[R_k^{1/\eta},R_k]$, so
\[
 \inf_{R_k^{1/\eta}\le r\le R_k}\frac{r^t}{m(r)}
 \ge R_k^{t-u}\longrightarrow\infty.
\]
Thus $I_\eta^t(K)=\infty$, proving \eqref{eq:homI}.
\end{proof}

\begin{remark}
\label{rem:homogeneous-endpoints}
The classical endpoint dimensions of $K$ are
\begin{equation}\label{eq:homogeneous-endpoint-dimensions}
 \dim_0^{\mathrm P}K=\dimH K=\liminf_{r\downarrow0}b(r),
 \qquad
 \dim_1^{\mathrm P}K=\dimP K=\limsup_{r\downarrow0}b(r).
\end{equation}
Indeed, the packing critical-exponent argument above applies also
at $\theta=0$ and $\theta=1$ in
\eqref{eq:homogeneous-outer-measure}. The windows are then
$I_0(R)=\{R\}$ and $I_1(R)=(0,R]$, respectively, and
\[
 \liminf_{R\downarrow0}\sup_{0<r\le R}b(r)
 =\lim_{R\downarrow0}\sup_{0<r\le R}b(r)
 =\limsup_{r\downarrow0}b(r).
\]
\end{remark}

\begin{remark}
\label{rem:homogeneous-countable-modification}
Under \eqref{eq:homogeneous},
\begin{equation}\label{eq:homogeneous-modified-intermediate}
 \overline{\dim}^{\mathrm{mod}}_{\eta}K
 =\overline{\dim}_{\eta}K\qquad(0<\eta\le1).
\end{equation}
To see this, put $\beta=\overline{\dim}_\eta K$.
For any $A\subset K$ with $\mu^*(A)>0$,
\eqref{eq:homogeneous-subset-covering-lower} and the preceding
critical-exponent argument give $I_\eta^t(A)=\infty$ whenever
$0\le t<\beta$. Thus
$\overline{\dim}_\eta A\ge\beta$; for $\beta=0$ this is immediate.
In every countable cover $K=\bigcup_j A_j$, some $A_j$ has
positive $\mu$-outer measure, by countable subadditivity.
Hence $\sup_j\overline{\dim}_\eta A_j\ge\beta$.
Taking the infimum over covers gives the lower bound in
\eqref{eq:homogeneous-modified-intermediate}, while the one-set
cover gives the reverse bound.
\end{remark}

\subsection{Digit sets and complementary pairs}

Write $\N=\{1,2,\ldots\}$ and $S^c=\N\setminus S$ for
$S\subset\N$. Define the binary digit set

\[
 E_S=\left\{\sum_{k=1}^\infty2^{-k}v_k:
 v_k\in\{0,1\}^d\text{ if }k\in S,\quad
 v_k=0\text{ if }k\notin S\right\}.
\]
Let $(U_k)_{k\ge1}$ be independent random vectors, each uniformly
distributed on $\{0,1\}^d$. The uniform digit measure $\mu_S$ is
the distribution of $\sum_{k\in S}2^{-k}U_k$. Write
\[
 A_S(0)=0,\qquad A_S(k)=\#(S\cap\{1,\ldots,k\})\quad(k\ge1).
\]
For $p\ge1$ and $N\in\N$, put
$\mathcal W_p(N)=\{N,\ldots,\lfloor pN\rfloor\}$.

\begin{proposition}
\label{prop:digit-mass-dimension-formulas}
For every $S\subset\N$, the set $E_S$ is nonempty and compact,
$\supp\mu_S=E_S$, and
\begin{equation}
\label{eq:complementary-digit-masses}
 \mu_S(B(x,r))\asymp_d2^{-dA_S(k)}
 \qquad(x\in E_S,\ k\ge0,\ 2^{-k-1}<r\le2^{-k}).
\end{equation}
The implicit constants depend only on $d$. The dimensions of $E_S$ are
\begin{align}
 \dim_\theta^{\mathrm P}E_S
 &=d\liminf_{N\to\infty}
       \max_{k\in\mathcal W_{1/(1-\theta)}(N)}\frac{A_S(k)}{k}
 &&(0\le\theta<1),\label{eq:complementary-digit-packing}\\
 \overline{\dim}^{\mathrm{mod}}_{\eta} E_S=\overline{\dim}_{\eta}E_S
 &=d\limsup_{N\to\infty}
       \min_{k\in\mathcal W_{1/\eta}(N)}\frac{A_S(k)}{k}
 &&(0<\eta\le1).\label{eq:complementary-digit-intermediate}
\end{align}
\end{proposition}

\begin{proof}
Continuity of the coding map and positivity of every cylinder
probability imply that $E_S$ is nonempty and compact and that
$\supp\mu_S=E_S$.
Each level-$k$ symbolic cylinder has probability $2^{-dA_S(k)}$
and image contained in a closed dyadic cube of side $2^{-k}$.
For $x\in E_S$ and $2^{-k-1}<r\le2^{-k}$, the ball $B(x,r)$ meets
at most $C_d$ such cubes. Choose $q=q(d)$ with
$\sqrt d\,2^{-q}<1/2$. A level-$(k+q)$ cylinder coding $x$ has
image contained in $B(x,r)$, so
\[
 2^{-dq}2^{-dA_S(k)}\le2^{-dA_S(k+q)}
 \le\mu_S(B(x,r))\le C_d2^{-dA_S(k)}.
\]
This proves \eqref{eq:complementary-digit-masses}, which in particular
implies the spatial comparison \eqref{eq:homogeneous} for $\mu_S$.

For a fixed $x_0\in E_S$ and real $a\ge1$, \eqref{eq:complementary-digit-masses} gives
\[
 \frac{\log\mu_S(B(x_0,2^{-a}))}{\log2^{-a}}
 =\frac{dA_S(\lfloor a\rfloor)}{a}+O_d(a^{-1}).
\]
Since $|A_S(k)-A_S(\ell)|\le|k-\ell|$ and $A_S(k)\le k$, for fixed
$p\ge1$ and $N=\lfloor a\rfloor$ we have
\begin{align*}
 \sup_{a\le t\le pa}\frac{dA_S(\lfloor t\rfloor)}{t}
 &=\max_{k\in\mathcal W_p(N)}\frac{dA_S(k)}{k}
   +O_{d,p}(a^{-1}),\\
 \inf_{a\le t\le pa}\frac{dA_S(\lfloor t\rfloor)}{t}
 &=\min_{k\in\mathcal W_p(N)}\frac{dA_S(k)}{k}
   +O_{d,p}(a^{-1}).
\end{align*}
Theorem~\ref{thm:homogeneous-window-formula} gives the packing
formula for $0<\theta<1$ and the unmodified intermediate formula.
The case $\theta=0$ in \eqref{eq:complementary-digit-packing}
follows from Remark~\ref{rem:homogeneous-endpoints}.
Remark~\ref{rem:homogeneous-countable-modification} gives the
modified equality in \eqref{eq:complementary-digit-intermediate}.
\end{proof}

\begin{remark}
\label{rem:digit-endpoint-dimensions}
In particular,
\begin{align}
 \dim_1^{\mathrm P}E_S=\dimP E_S=\udimB E_S
 &=d\limsup_{k\to\infty}\frac{A_S(k)}{k},\label{eq:complementary-digit-packing-endpoint}\\
 \dimH E_S=\ldimB E_S
 &=d\liminf_{k\to\infty}\frac{A_S(k)}{k}.\nonumber
\end{align}
Indeed, \eqref{eq:complementary-digit-masses} and
Remark~\ref{rem:homogeneous-endpoints} give the Hausdorff and
packing dimension formulas. For the box dimensions, we have
\begin{equation}\label{eq:digit-covering-number}
 N_{2^{-k}}(E_S)\asymp_d2^{dA_S(k)},
\end{equation}
where $N_r$ is the least number of open balls of radius $r$
needed to cover the set. The upper bound follows by covering
each of the $2^{dA_S(k)}$ level-$k$ cylinder images with a number
of radius-$2^{-k}$ balls bounded in terms of $d$. For the lower
bound, every such ball meeting $E_S$ is contained in a ball
of radius $2^{1-k}$ centred in $E_S$; by
\eqref{eq:complementary-digit-masses}, it has mass at most
$C_d2^{-dA_S(k)}$. Thus any cover requires at least
$c_d2^{dA_S(k)}$ balls. Taking the lower and upper limits of
logarithmic ratios in \eqref{eq:digit-covering-number} proves
the two box-dimension formulas.
\end{remark}

Since $A_{S^c}(k)=k-A_S(k)$, the window formulas give the
following complementary identities.

\begin{lemma}
\label{lem:complementary-digits}
Let $S\subset\N$, put $F=E_{S^c}$, and let
$\mu=\mu_S$ and $\nu=\mu_{S^c}$ be the uniform digit measures.
Then
\begin{equation}\label{eq:complementary-digit-dimensions}
 \begin{split}
 &\overline{\dim}_{1-\theta}F
 =\overline{\dim}^{\mathrm{mod}}_{1-\theta}F
 =d-\dim_\theta^{\mathrm P}E_S
 \quad(0\le\theta\le1),\\
 &\dimH(E_S\times F)=d.
 \end{split}
\end{equation}
Moreover, $\mu\times\nu$ is $d$-Ahlfors regular on $E_S\times F$.
\end{lemma}
\begin{proof}
Since $S^c=\N\setminus S$, we have $A_{S^c}(k)=k-A_S(k)$.
For $0\le\theta<1$ and $p=(1-\theta)^{-1}$,
\eqref{eq:complementary-digit-intermediate} applied to $F$ gives
\begin{align*}
 \overline{\dim}_{1-\theta}F
 &=d\limsup_{N\to\infty}
     \min_{k\in\mathcal W_p(N)}\left(1-\frac{A_S(k)}{k}\right)\\
 &=d-d\liminf_{N\to\infty}
     \max_{k\in\mathcal W_p(N)}\frac{A_S(k)}{k}\\
 &=d-\dim_\theta^{\mathrm P}E_S.
\end{align*}
The endpoint $\theta=1$ follows from
Remark~\ref{rem:digit-endpoint-dimensions} and the Hausdorff
convention at intermediate parameter zero. The modified identity
follows from \eqref{eq:complementary-digit-intermediate} for
$\theta<1$ and the same convention for $\theta=1$.

The inclusions
\[
 B(x,r/\sqrt2)\times B(y,r/\sqrt2)
 \subset B((x,y),r)\subset B(x,r)\times B(y,r)
\]
and \eqref{eq:complementary-digit-masses}, also at the comparable
radius $r/\sqrt2$, give
\[
 (\mu\times\nu)(B((x,y),r))
 \asymp_d2^{-dA_S(k)}2^{-d(k-A_S(k))}=2^{-dk}\asymp_dr^d
\]
for $(x,y)\in E_S\times F$ and $2^{-k-1}<r\le2^{-k}$.
Hence $\mu\times\nu$ is $d$-Ahlfors regular and
$\dimH(E_S\times F)=d$.

\end{proof}

The same allocation of digits also gives the additive identities
\begin{equation}\label{eq:complementary-digit-convolution}
 E_S+F=[0,1]^d,\qquad
 \mu*\nu=\mathcal L^d\!\restriction[0,1]^d.
\end{equation}
Using the independent random vectors $(U_k)$ from the definition,
put
\[
 X=\sum_{k\in S}2^{-k}U_k,\qquad
 Y=\sum_{k\in S^c}2^{-k}U_k.
\]
Then $X$ and $Y$ are independent with distributions $\mu$ and $\nu$,
and $X+Y=\sum_{k\ge1}2^{-k}U_k$ is uniformly distributed on
$[0,1]^d$. This proves the convolution identity; the sumset identity
follows from $E_S+F=\supp(\mu*\nu)$.

\subsection{Classification of attainable profiles}

For a function $f:[0,1]\to[0,d]$, write $D^+f$ for its upper right
Dini derivative. Consider the differential condition
\begin{equation}\label{eq:dini}
 0\le D^+f(\theta)
 \le\frac{f(\theta)(d-f(\theta))}{d(1-\theta)}
 \qquad(0\le\theta<1).
\end{equation}

\begin{theorem}
\label{thm:attain}
Let $d\ge1$ be an integer and let $f:[0,1]\to[0,d]$.
There is a nonempty bounded set $E\subset\R^d$ with
$\dim_\theta^{\mathrm P}E=f(\theta)$ for every $\theta\in[0,1]$
if and only if $f$ is nondecreasing on $[0,1]$, continuous on
$[0,1)$, and satisfies \eqref{eq:dini}.
\end{theorem}

For sufficiency, we use the homogeneous Moran realization of
Banaji and Rutar~\cite[Theorem~3.9]{BanajiRutar2022}.
The next lemma transfers its upper intermediate profile and Hausdorff
endpoint to a binary digit set. Complementary digits then give the
required packing spectrum.

\begin{lemma}
\label{lem:moran-digit-transfer}
Let $d\ge1$ be an integer, and let $C\subset[0,1]^d$ be a
homogeneous Moran set as in \cite[Section~2.3]{BanajiRutar2022}:
at level $j$, each construction cube is replaced by its $2^d$
corner subcubes with common contraction ratio $r_j\in(0,1/2]$.
There exists $T\subset\N$ such that
\[
 \overline{\dim}_{\eta}E_T=\overline{\dim}_{\eta}C
 \quad(0<\eta\le1),
 \qquad \dimH E_T=\dimH C.
\]
\end{lemma}

\begin{proof}
Put $\rho_0=1$, $\rho_j=r_1\cdots r_j$, and
$\tau_j=-\log_2\rho_j$.  Define
\[
 T=\{\lfloor\tau_j\rfloor:j\ge1\}.
\]
For real $t>0$, let $j(t)$ be the unique integer satisfying
$\tau_{j(t)-1}<t\le\tau_{j(t)}$.
Write
\[
 \sigma(2^{-t})=\frac{d\,j(t)}{t}.
\]
For $t\ge1$, the digit count satisfies
\[
 A_T(\lfloor t\rfloor)=\#\{j\ge1:\tau_j<\lfloor t\rfloor+1\},
 \qquad j(t)=1+\#\{j\ge1:\tau_j<t\}.
\]
The interval $[t,\lfloor t\rfloor+1)$ contains at most one $\tau_j$.
Consequently,
\begin{equation}\label{eq:moran-digit-count-error}
 A_T(\lfloor t\rfloor)-j(t)\in\{-1,0\},
 \qquad
 \left|\frac{dA_T(\lfloor t\rfloor)}{t}-\sigma(2^{-t})\right|
 \le\frac{d}{t}.
\end{equation}
Fix $x\in E_T$. The digit mass estimate
\eqref{eq:complementary-digit-masses} therefore gives
\begin{equation}\label{eq:moran-digit-mass-error}
 \frac{\log\mu_T(B(x,2^{-t}))}{\log2^{-t}}
 =\sigma(2^{-t})+O_d(t^{-1}).
\end{equation}

By \cite[Proposition~2.11]{BanajiRutar2022},
\[
 \overline{\dim}_{\eta}C
 =\limsup_{R\downarrow0}\inf_{R^{1/\eta}\le r\le R}\sigma(r)
 \quad(0<\eta\le1),
 \qquad
 \dimH C=\liminf_{r\downarrow0}\sigma(r).
\]
Their construction uses the maximum norm; these dimension values
are unchanged under the equivalent Euclidean norm used here.
For each fixed $0<\eta\le1$, write $R=2^{-a}$ with
$a\to\infty$. On the entire window $a\le t\le a/\eta$,
the error in \eqref{eq:moran-digit-mass-error} is $O_d(a^{-1})$.
Taking window infima and then upper limits, and applying
\eqref{eq:homI} to $E_T$, proves the upper intermediate identity.
Taking lower limits over all $t\to\infty$ and using
\eqref{eq:homogeneous-endpoint-dimensions} proves the Hausdorff identity.
\end{proof}

\begin{proof}[Proof of Theorem~\ref{thm:attain}]
Necessity follows from Remark~\ref{rem:parameter-monotonicity}
and Corollary~\ref{cor:packing-spectrum-regularity}.
For sufficiency, suppose that $f$ satisfies the stated conditions
and put $h(\eta)=d-f(1-\eta)$ for $0\le\eta\le1$.
The locally bounded Dini estimate \eqref{eq:dini} and the Dini
mean-value principle \cite[Corollary~2.3]{BanajiRutar2022} imply
that $f$ is Lipschitz on every compact subinterval of $[0,1)$.
Thus $h$ is nondecreasing, continuous on $(0,1]$, and locally
absolutely continuous there, with
\[
 0\le h'(\eta)=f'(1-\eta)
 \le\frac{h(\eta)(d-h(\eta))}{d\eta}
 \quad\text{for almost every }\eta\in(0,1).
\]
Integrating over $[\eta,\eta+u]$, dividing by $u$, and letting
$u\downarrow0$ gives the same upper bound for $D^+h(\eta)$
at every $0<\eta<1$.
Hence $h\in\mathcal H(0,d)$ in the notation of
\cite[Definition~1.3]{BanajiRutar2022}.

By \cite[Theorem~3.9]{BanajiRutar2022}, with $\lambda=0$ and
$\alpha=d$, there is a homogeneous Moran set $C$ with
contraction ratios $r_j\in(0,1/2]$ such that
\[
 \overline{\dim}_{\eta}C=h(\eta)\quad(0<\eta\le1),
 \qquad \dimH C=h(0).
\]
Lemma~\ref{lem:moran-digit-transfer} supplies a binary digit set
$F=E_T$ with these same values. Put $S=\N\setminus T$ and
$E=E_S$. Lemma~\ref{lem:complementary-digits} gives
\[
 \dim_\theta^{\mathrm P}E
 =d-\overline{\dim}_{1-\theta}F
 =d-h(1-\theta)=f(\theta)
 \qquad(0\le\theta\le1).
\]
\end{proof}

\begin{remark}
\label{rem:digit-profile-realization}
The proof of Theorem~\ref{thm:attain} gives a stronger realization:
for every $f$ satisfying its conditions, there is a compact binary
digit set $E_S\subset[0,1]^d$ with
$\dim_\theta^{\mathrm P}E_S=f(\theta)$ for all $\theta\in[0,1]$.
For the realizing set $E_S$, the uniform digit measure also satisfies
\[
 d_\theta^{\mu_S}(x)=f(\theta)
 \qquad(x\in E_S,\ 0\le\theta<1),
\]
where the windowed local exponent is defined in
Definition~\ref{def:windowed-local-dimension}. This follows from
\eqref{eq:complementary-digit-masses} and
\eqref{eq:complementary-digit-packing}.
Its classical upper local dimension satisfies
\[
 \limsup_{r\downarrow0}
 \frac{\log\mu_S(B(x,r))}{\log r}
 =d\limsup_{k\to\infty}\frac{A_S(k)}{k}
 =f(1)\qquad(x\in E_S),
\]
by \eqref{eq:complementary-digit-masses} and
\eqref{eq:complementary-digit-packing-endpoint}. Thus the same measure
also attains the classical packing variational supremum.

The complementary set satisfies
\[
 \overline{\dim}_{\eta}E_{S^c}
 =\overline{\dim}^{\mathrm{mod}}_{\eta}E_{S^c}
 =d-f(1-\eta)\qquad(0\le\eta\le1).
\]
Thus this pair attains the Hausdorff product bound at every parameter.
\end{remark}

\begin{corollary}
\label{cor:reflected-profile-classification}
Let $d\ge1$ be an integer and let $f:[0,1]\to[0,d]$.
Then $f$ is an attainable packing spectrum in $\R^d$ if and only if
$h(\eta)=d-f(1-\eta)$ is an attainable upper intermediate-dimension
profile in $\R^d$.
\end{corollary}

\begin{proof}
If $f$ is an attainable packing spectrum, Theorem~\ref{thm:attain}
and Remark~\ref{rem:digit-profile-realization} give a complementary
digit set with upper intermediate profile $h$.
Conversely, if $h$ is an attainable
upper intermediate profile,
\cite[Theorem~A]{BanajiRutar2022} gives $h\in\mathcal H(0,d)$.
By the Dini mean-value principle \cite[Corollary~2.3]{BanajiRutar2022}, $h$ is Lipschitz on every
$[a,1]$ with $a>0$. Reflection therefore gives a locally
absolutely continuous function $f(\theta)=d-h(1-\theta)$ on
$[0,1)$ with the required differential bound almost everywhere.
For $0\le\theta<\theta+u<1$, absolute continuity gives
\[
 \begin{aligned}
 f(\theta+u)-f(\theta)
 =\int_{1-\theta-u}^{1-\theta}h'(v)\,dv
 \le\int_\theta^{\theta+u}
       \frac{f(w)(d-f(w))}{d(1-w)}\,dw.
 \end{aligned}
\]
Dividing by $u$ and letting $u\downarrow0$ recovers
\eqref{eq:dini} at every $\theta<1$.
Monotonicity and continuity on $[0,1)$ are preserved,
so Theorem~\ref{thm:attain} applies to $f$.
\end{proof}

\section{Examples and further questions}
\label{sec:further-questions}

\subsection{Consequences and limitations}

\begin{example}[Exact additivity beyond the endpoint bounds]
\label{ex:strict-product-improvement-proof}
\label{ex:strict-product-improvement}
There are compact digit sets $E,F\subset[0,1]$, each with distinct
Hausdorff and packing dimensions, for which the product is
Hausdorff-additive:
\[
 \dimH E=\dimH F=\frac14,\qquad
 \dimP E=1,\qquad \dimP F=\udimB F=\frac12,
 \qquad \dimH(E\times F)=\frac12.
\]
The complementary-window bound gives
\[
 \frac12\le\dimH(E\times F)\le\frac34-\frac{\theta}{4}
 \qquad(0<\theta<1).
\]
Letting $\theta\uparrow1$ gives the value $1/2$, whereas the classical
Hausdorff--packing bound gives only
\[
 \dimH(E\times F)
 \le\min\left\{\frac14+\frac12,\,1+\frac14\right\}
 =\frac34.
\]
\end{example}

\begin{proof}
By Remark~\ref{rem:digit-profile-realization}, choose a compact digit
set $E$ with $\dim_\theta^{\mathrm P}E=1/4$ for $0\le\theta<1$
and $\dim_1^{\mathrm P}E=1$. The function
$g(\theta)=1/2+\theta/4$ also satisfies the realization hypotheses:
its Dini bound reduces to $\theta(4-\theta)\ge0$ for
$0\le\theta<1$. Let $E_T$ realize $g$ and put $F=E_{T^c}$.
Lemma~\ref{lem:complementary-digits} gives
$\overline{\dim}_{\eta}F=(1+\eta)/4$ for $0\le\eta\le1$ and
$\dimP F=\udimB F=1/2$. Thus
Theorem~\ref{thm:direct-ffk-hausdorff-product} gives the displayed
upper bound for $0<\theta<1$, while the classical Hausdorff product
lower bound gives $\dimH(E\times F)\ge1/2$. Passing to the limit
proves the asserted additivity.
\end{proof}

\begin{example}[Equal endpoint dimensions, different geometry]
\label{ex:extreme-profiles}
Let $d\ge1$ be an integer and fix $0<h<d$. The profiles
\[
 f_{\mathrm{flat}}(\theta)=
 \begin{cases}
 h,&0\leq\theta<1,\\
 d,&\theta=1,
 \end{cases}
 \qquad
 f_{\mathrm{fast}}(\theta)
 =\frac{dh}{d(1-\theta)+\theta h}
 \quad(0\leq\theta\leq1)
\]
are realized by compact digit sets in $[0,1]^d$ with the same
Hausdorff dimension $h$ and packing dimension $d$. The flat profile
has a jump of size $d-h$ at the packing endpoint. The fast profile
attains the upper bound in \eqref{eq:dim_inequality} for every pair
$0\le\theta_1<\theta_2<1$, proving sharpness of the full
two-parameter estimate. Moreover, every $\alpha$-H\"older surjection
from a flat realization onto a fast realization satisfies
$\alpha\le h/d$, although the separate Hausdorff and packing
bounds give only $\alpha\le1$.
\end{example}

\begin{proof}
Both profiles are nondecreasing and continuous on $[0,1)$, with
$f_{\mathrm{flat}}'=0$ there and
\[
 f_{\mathrm{fast}}'(\theta)
 =\frac{f_{\mathrm{fast}}(\theta)
        (d-f_{\mathrm{fast}}(\theta))}{d(1-\theta)}
 \qquad(0\le\theta<1).
\]
Theorem~\ref{thm:attain} and
Remark~\ref{rem:digit-profile-realization} therefore supply the digit
realizations. Direct substitution of $f_{\mathrm{fast}}$ into
\eqref{eq:dim_inequality} gives equality for every pair of parameters.
For a H\"older surjection $E_{\mathrm{flat}}\to E_{\mathrm{fast}}$,
the classical Hausdorff bound first gives $\alpha\le1$.
Proposition~\ref{prop:basic-properties}\ref{item:basic-mappings}
then yields
$f_{\mathrm{fast}}(\theta)\le h/\alpha$ for $\theta<1$;
letting $\theta\uparrow1$ gives $\alpha\le h/d$.
\end{proof}

For the uniform digit measure $\mu$ on a flat realization,
Remark~\ref{rem:digit-profile-realization} gives, at every
$x\in\supp\mu$,
\[
 \lim_{\theta\uparrow1}d_\theta^\mu(x)
 =h<d
 =\limsup_{r\downarrow0}\frac{\log\mu(B(x,r))}{\log r}.
\]
Hence the finite-window local exponents do not in general converge
to the classical upper local dimension at the packing endpoint.

\begin{example}[Compactness does not ensure attainment]
\label{ex:add-nonattainment}
Although positive critical outer measure guarantees attainment in
Theorem~\ref{thm:main} by Corollary~\ref{cor:positive-content},
compactness does not. There is a compact set $K\subset[0,1]$ such
that, for every $0<\theta<1$,
\[
 \dim_\theta^{\mathrm P}K=1,
 \qquad \mathcal P^1_\theta(K)=0,
\]
but every $\mu\in\mathcal M_1(K)$ satisfies
\[
 \operatorname*{ess\,inf}_{x\sim\mu}d_\theta^\mu(x)<1.
\]
\end{example}

\begin{proof}
For $j\ge1$, choose a compact set $C_j\subset[0,1]$ with
\[
 \dimH C_j=\dimP C_j=D_j:=\frac{j}{j+1}\uparrow1;
\]
for example, take the two-branch Cantor set with contraction ratio
$2^{-1/D_j}$. Put
\[
 K_j=2^{-j}+2^{-j-2}C_j,
 \qquad K=\{0\}\cup\bigcup_{j\ge1}K_j.
\]
The sets $K_j$ lie in pairwise disjoint closed intervals accumulating
at $0$, so $K$ is compact. For every $0<\theta<1$, the endpoint
bounds and bi-Lipschitz invariance give
$\dim_\theta^{\mathrm P}K_j=D_j$. Countable stability gives
$\dim_\theta^{\mathrm P}K=1$, while the critical-exponent property
and countable subadditivity give $\mathcal P^1_\theta(K)=0$.

Fix $0<\theta<1$ and $\mu\in\mathcal M_1(K)$.
If $\mu(\{0\})>0$, then $d_\theta^\mu(0)=0$.
Otherwise $\mu(K_j)>0$ for some $j$, and
Corollary~\ref{cor:measure-spectra} gives
\[
 \operatorname*{ess\,inf}_{\mu}d_\theta^\mu
 =P^-_\theta(\mu)
 \le\dim_\theta^{\mathrm P}K_j=D_j<1.
\]
Thus no probability measure attains the supremum.
\end{proof}

\subsection{Comparison with other intermediate dimensions}

For a compactly supported Borel probability measure $\nu$ and
$0<\theta\le1$, write $\dim_\theta^{\mathrm{AP}}\nu$ for
Angelini's packing intermediate dimension.
This is denoted by $\dim_{P,\theta}\nu$ in
\cite{Angelini2026c}. We use the ball characterization
\cite[Theorem~4.3]{Angelini2026c}:
\begin{equation}\label{eq:angelini-measure-ball}
 \dim_\theta^{\mathrm{AP}}\nu
 =\sup\left\{s\ge0:
   \liminf_{R\downarrow0}
   \sup_{R^{1/\theta}\le r\le R}
      \frac{\nu(B(x,r))}{r^s}=0
   \text{ for $\nu$-almost every }x\right\}.
\end{equation}
For a nonempty analytic set $K$, Angelini's set invariant is
\cite[Definition~8.2]{Angelini2026c}
\begin{equation}\label{eq:angelini-set-definition}
 \dim_\theta^{\mathrm{AP}}K
 :=\sup_{\substack{\nu\text{ probability}\\\supp\nu\Subset K}}
       \dim_\theta^{\mathrm{AP}}\nu
 \qquad(0<\theta\le1).
\end{equation}
At $\theta=1$, this gives $\dimP K$.

\begin{proposition}
\label{rem:angelini-homogeneous-separation}
\label{prop:angelini-homogeneous-comparison}
Let $\mu$ be a Borel probability measure with
$\supp\mu=K\subset\R^d$ satisfying \eqref{eq:homogeneous}. Then
\begin{equation}
\label{eq:angelini-homogeneous-intermediate}
 \dim_\theta^{\mathrm{AP}}K
 =\overline{\dim}_\theta K
 =\overline{\dim}^{\mathrm{mod}}_\theta K
 \qquad(0<\theta\le1).
\end{equation}
\end{proposition}

\begin{proof}
Fix $x_0\in K$, use the mass function $m(r)$ and exponent $b(r)$
from \eqref{eq:homogeneous-mass-notation}, and put
$\beta=\limsup_{R\downarrow0}\inf_{R^{1/\theta}\le r\le R}b(r)$.
For $0\le s<t<\beta$, arbitrarily small windows have $b(r)>t$
throughout, so spatial comparison gives
\[
\sup_{R^{1/\theta}\le r\le R}\frac{\mu(B(x,r))}{r^s}
 \le LR^{t-s}\longrightarrow0
 \qquad(x\in K)
\]
along those windows. Thus $\dim_\theta^{\mathrm{AP}}\mu\ge\beta$.

For the reverse bound, let $\nu\in\mathcal M_1(K)$ be arbitrary.
The differentiation theorem for Radon measures
\cite[Theorem~2.12(1)]{Mattila1995} gives, for $\nu$-almost every $x$,
a constant $C_x\ge1$ such that, for all sufficiently small $r$,
\[
 \mu(B(x,r))\le C_x\nu(B(x,r)),
 \qquad
 \nu(B(x,r))\ge(LC_x)^{-1}m(r).
\]
If $\beta<t<s$, every sufficiently small window contains a radius
with $b(r)<t$, and hence
\[
\sup_{R^{1/\theta}\le r\le R}\frac{\nu(B(x,r))}{r^s}
 \ge(LC_x)^{-1}R^{t-s}\longrightarrow\infty
 \quad\text{for $\nu$-almost every }x.
\]
Thus \eqref{eq:angelini-measure-ball} gives
$\dim_\theta^{\mathrm{AP}}\nu\le\beta$ for every $\nu\in\mathcal M_1(K)$.
Taking the supremum in \eqref{eq:angelini-set-definition} and applying
Theorem~\ref{thm:homogeneous-window-formula} gives
$\dim_\theta^{\mathrm{AP}}K=\overline{\dim}_\theta K$.
Remark~\ref{rem:homogeneous-countable-modification} supplies the
modified equality in \eqref{eq:angelini-homogeneous-intermediate}.
\end{proof}

\begin{example}[Incomparability with other intermediate dimensions]
\label{ex:comparison-constructions}
\label{ex:add-geometric-block-example}
The examples below distinguish our packing spectrum from Angelini's
packing intermediate dimension. Under the spatial mass comparison
assumption \eqref{eq:homogeneous}, Angelini's set invariant nevertheless
coincides with both the upper intermediate dimension and its countably
stable modification. All the digit sets below satisfy this assumption.

\emph{Geometric blocks.} Let
\[
 S=\bigcup_{j=0}^{\infty}
       \{4^j,4^j+1,\ldots,2\cdot4^j-1\},\qquad E=E_S\subset[0,1].
\]
Then
\begin{align}
 \dim_\theta^{\mathrm P}E
 &=\min\left\{\frac{1}{3-2\theta},\frac{2}{3}\right\}
 &&(0\le\theta\le1),\label{eq:add-block-packing}\\
 \dim_\theta^{\mathrm{AP}}E
 =\overline{\dim}_\theta E
 =\overline{\dim}^{\mathrm{mod}}_\theta E
 &=\max\left\{\frac{1}{3},\frac{2\theta}{1+2\theta}\right\}
 &&(0<\theta\le1).\label{eq:add-block-intermediate}
\end{align}

\emph{Flat profiles.} For every integer $d\ge1$ and $0\le h<d$,
there is a compact digit set $K_h\subset[0,1]^d$ such that
\begin{equation}\label{eq:maximal-construction-separation}
 \begin{aligned}
 \dim_\theta^{\mathrm P}K_h&=h,\\
 \dim_\theta^{\mathrm{AP}}K_h
 =\overline{\dim}_\theta K_h
 =\overline{\dim}^{\mathrm{mod}}_\theta K_h&=d
 \qquad(0<\theta<1).
 \end{aligned}
\end{equation}
\end{example}

\begin{proof}
For the first part, define $a:(0,\infty)\to[1/3,2/3]$ by
$a(4t)=a(t)$ and
\[
 a(t)=
 \begin{cases}
  1-\dfrac{2}{3t},&1\le t\le2,\\[4pt]
  \dfrac{4}{3t},&2\le t\le4.
 \end{cases}
\]
For each fixed $p\ge1$, the explicit formulas for $A_S$ and $a$ give,
uniformly for $t\in[1,4]$,
\begin{align*}
\max_{k\in\mathcal W_p(\lfloor4^j t\rfloor)}
\frac{A_S(k)}{k}
&=\max_{t\le u\le pt}a(u)+O_p(4^{-j}),\\
\min_{k\in\mathcal W_p(\lfloor4^j t\rfloor)}
\frac{A_S(k)}{k}
&=\min_{t\le u\le pt}a(u)+O_p(4^{-j}).
\end{align*}

Consequently,
\[
 \begin{aligned}
 \liminf_{N\to\infty}\max_{k\in\mathcal W_p(N)}\frac{A_S(k)}{k}
 &=\min_{1\le t\le4}\max_{t\le u\le pt}a(u)=:M(p),\\
 \limsup_{N\to\infty}\min_{k\in\mathcal W_p(N)}\frac{A_S(k)}{k}
 &=\max_{1\le t\le4}\min_{t\le u\le pt}a(u)=:m(p).
 \end{aligned}
\]
For $1/3\le b<2/3$, the connected components of $\{u:a(u)\le b\}$ are
\[
 4^j\left[\frac{4}{3b},\frac{8}{3(1-b)}\right]
 \qquad(j\in\mathbb Z).
\]
Their endpoint ratio is $2b/(1-b)$, so one contains a window
$[t,pt]$ if and only if $p\le2b/(1-b)$.
Using also $a(2u)=1-a(u)$, we obtain
\[
 M(p)=\min\left\{\frac{p}{p+2},\frac23\right\},
 \qquad
 m(p)=1-M(p)=\max\left\{\frac{2}{p+2},\frac13\right\}.
\]
Proposition~\ref{prop:digit-mass-dimension-formulas}, with
$p=(1-\theta)^{-1}$ and $p=\theta^{-1}$, gives the packing and
upper intermediate profiles, including the modified profile.
Remark~\ref{rem:digit-endpoint-dimensions} gives $\dimH E=1/3$
and $\dimP E=2/3$, supplying their endpoints.
The uniform digit measure satisfies
\eqref{eq:homogeneous}, so Proposition~\ref{rem:angelini-homogeneous-separation}
identifies Angelini's profile with the upper intermediate profile.

For the second part, Remark~\ref{rem:digit-profile-realization}
gives a digit set $K_h=E_T\subset[0,1]^d$ with the flat profile of
Example~\ref{ex:extreme-profiles}, including when $h=0$. Its
complementary digit set $L_h=E_{T^c}$ satisfies
\[
 \dimH L_h=d-\dimP K_h=0.
\]
Corollary~\ref{cor:packing-spectrum-regularity} therefore gives
$\dim_\alpha^{\mathrm P}L_h=0$ for every $0\le\alpha<1$.
Applying Lemma~\ref{lem:complementary-digits} to $L_h$ and
then Proposition~\ref{rem:angelini-homogeneous-separation} gives \eqref{eq:maximal-construction-separation}.
\end{proof}

\subsection{Further questions}

Complementary digit pairs attain the product bound at every
parameter by Lemma~\ref{lem:complementary-digits}.
Example~\ref{ex:strict-product-improvement} attains the optimized
bound as $\theta\uparrow1$.

\begin{question}\label{question:sharp-product-pairs}
Which ordered pairs of nonempty compact sets $E\subset\R^m$ and
$F\subset\R^n$ satisfy
\[
 \dimH(E\times F)=\inf_{0\le\theta\le1}
 \left\{\dim_\theta^{\mathrm P}E+\overline{\dim}_{1-\theta}F\right\}?
\]
\end{question}

Complementary digit sets satisfy the stronger identity
$E_S+E_{S^c}=[0,1]^n$, motivating an additive version of product
testing. For a nonempty compact set $E\subset\R^n$, put
\[
 \mathcal C(E)=\{B\subset\R^n:B\text{ nonempty and compact},
 \ \operatorname{int}(E+B)\ne\varnothing\}.
\]
Theorem~\ref{thm:direct-ffk-hausdorff-product} and Lipschitz
addition give
\begin{equation}\label{eq:additive-complement-bound}
 \dim_\theta^{\mathrm P}E
 \ge n-\inf_{B\in\mathcal C(E)}\overline{\dim}_{1-\theta}B
 \qquad(0\le\theta\le1),
\end{equation}
and the analogous bound with
$\overline{\dim}^{\mathrm{mod}}_{1-\theta}B$.
Equality at both endpoints, including for the modified dimensions,
follows from Keleti--M\'ath\'e~\cite[Theorems~1.7--1.8]{KeletiMathe2023}.
For every binary digit set, Lemma~\ref{lem:complementary-digits} and
\eqref{eq:complementary-digit-convolution} give equality at all
parameters with the single complement $B=E_{S^c}$.
The product tester supplied by
Lemma~\ref{lem:dyadic-tester-construction} does not settle the
additive problem: $\dimH\Gamma\ge n$ for a subset
$\Gamma\subset E\times B$ does not imply
$\operatorname{int}(E+B)\ne\varnothing$.

\begin{question}\label{question:additive-complement}
For every nonempty compact $E\subset\R^n$ and every $0<\theta<1$,
is
\[
 \dim_\theta^{\mathrm P}E
 =n-\inf_{B\in\mathcal C(E)}\overline{\dim}_{1-\theta}B?
\]
Does equality also hold with
$\overline{\dim}^{\mathrm{mod}}_{1-\theta}B$ in place of
$\overline{\dim}_{1-\theta}B$?
\end{question}

Proposition~\ref{prop:angelini-homogeneous-comparison} proves the
equality below under the spatial mass comparison assumption
\eqref{eq:homogeneous}.

\begin{question}\label{question:angelini-modified}
For which natural classes of nonempty analytic sets $E\subset\R^d$
does
\[
 \dim_\theta^{\mathrm{AP}}E
 =\overline{\dim}^{\mathrm{mod}}_\theta E
 \qquad(0<\theta<1)
\]
hold? In particular, can the spatial mass comparison assumption
\eqref{eq:homogeneous} be replaced by weaker geometric or
measure-theoretic conditions?
\end{question}
Such an extension would give a variational characterization of
modified upper intermediate dimension on these classes of sets
in terms of Angelini's measure invariant.

The monotonicity of the packing spectrum gives the limiting invariant
\begin{equation}\label{eq:left-packing-limit}
 \dim^{\mathrm{qP}}E
 :=\lim_{\theta\uparrow1}\dim_\theta^{\mathrm P}E
 =\sup_{0\le\theta<1}\dim_\theta^{\mathrm P}E.
\end{equation}
Proposition~\ref{prop:basic-properties} and
Corollary~\ref{cor:packing-spectrum-regularity} give
\[
 \dimH E\le\dim^{\mathrm{qP}}E\le\dimP E.
\]
Taking suprema in the corresponding spectral properties shows that
$\dim^{\mathrm{qP}}$ is countably stable and
bi-Lipschitz invariant, satisfies
\[
 \dim^{\mathrm{qP}}f(E)
 \le\alpha^{-1}\dim^{\mathrm{qP}}E
\]
for every $\alpha$-H\"older map $f$ with $0<\alpha\le1$, and is
compactly determined on nonempty analytic sets:
\[
 \dim^{\mathrm{qP}}E
 =\sup_{\substack{\varnothing\ne K\subset E\\K\text{ compact}}}
       \dim^{\mathrm{qP}}K.
\]
The flat and fast realizations in Example~\ref{ex:extreme-profiles}
have limiting values $h$ and $d$, respectively, despite their
identical Hausdorff and packing dimensions.

For a nonempty bounded set $F$, define
\[
 \overline{\dim}_{0+}F
 :=\lim_{\eta\downarrow0}\overline{\dim}_{\eta}F,
 \qquad
 \overline{\dim}^{\mathrm{mod}}_{0+}F
 :=\lim_{\eta\downarrow0}\overline{\dim}^{\mathrm{mod}}_{\eta}F.
\]
A positive gap $\dimP E-\dim^{\mathrm{qP}}E$ forces a jump at
$\eta=0$ in the upper intermediate dimensions of every sufficiently
accurate packing-dimension testing set.
\begin{proposition}
\label{rem:left-packing-limit}
\label{prop:tester-endpoint-jumps}
Let $E\subset\R^n$ be nonempty and analytic, let
$F\subset[0,1]^n$ be nonempty and compact, and let $\varepsilon\ge0$.
If
\[
 \dimH(E\times F)-\dimH F\ge\dimP E-\varepsilon,
\]
then
\begin{equation}\label{eq:tester-endpoint-jump}
 \begin{split}
 \overline{\dim}_{0+}F-\dimH F
 \ge\overline{\dim}^{\mathrm{mod}}_{0+}F-\dimH F
 \ge\dimP E-\dim^{\mathrm{qP}}E-\varepsilon.
 \end{split}
\end{equation}
\end{proposition}

\begin{proof}
Letting $\theta\uparrow1$ in the modified assertion of
Theorem~\ref{thm:direct-ffk-hausdorff-product} gives
\[
 \dimH(E\times F)
 \le\dim^{\mathrm{qP}}E
      +\overline{\dim}^{\mathrm{mod}}_{0+}F.
\]
Combine this with the hypothesis and
$\overline{\dim}^{\mathrm{mod}}_{0+}F\le\overline{\dim}_{0+}F$.
\end{proof}

It remains to characterize $\dim^{\mathrm{qP}}$ directly through
covers or measures and to determine geometric criteria for
$\dim^{\mathrm{qP}}E=\dimP E$.

\section*{Acknowledgements}
The author thanks Lars Olsen for the invitation to visit the Analysis
Group at the University of St Andrews, where part of this work was
carried out, and Kenneth Falconer for helpful discussions. The visit
was supported by the China Scholarship Council
(award No.\ 202506840058).

\end{document}